\documentclass[12pt,reqno]{article}
\usepackage{amsmath, amsthm, amssymb, graphicx}
\newcommand{\eps}{\epsilon}
\newcommand{\stwo}{\sqrt{2}}
\newcommand{\rstwo}{{1\over\stwo}}
\newcommand{\Qll}{Q_{[1,1]}}
\newcommand{\Qlo}{Q_{[1,0]}}
\newcommand{\Ill}{I_{[1,1]}}
\newcommand{\Ilo}{I_{[1,0]}}
\newcommand{\Iol}{I_{[0,1]}}
\newcommand{\Qol}{Q_{[0,1]}}

\newcommand{\om}{\omega}
\newcommand{\into}{\int_\Omega}

\newcommand{\Om}{\Omega}

\newcommand{\RR}{{\mathbb R}}
\newcommand{\ZZ}{{\mathbb Z}}
\newcommand{\NN}{{\mathbb N}}
\newcommand{\CC}{{\mathbb C}}
\newcommand{\la}{\lambda}
\newcommand{\HH}{{\bf H}}
\newcommand{\Pf}{\noindent {\it Proof: \  }}
\newcommand{\QED}{\newline $\diamondsuit$}

\renewcommand{\Im}{{\rm Im\,}}
\renewcommand{\Re}{{\rm Re\,}}

\newcommand{\LL}{{\cal L}}
\newcommand{\XL}{{\mathcal X}_{\lambda}}
\newcommand{\XXo}{{\mathcal X}_{\lambda_0}}
\newcommand{\lan}{\langle}
\newcommand{\ran}{\rangle}
\newcommand{\llan}{\left\lan}
\newcommand{\rran}{\right\ran}

\newcommand{\mul}{\mu_1(\lambda)}
\newcommand{\lab}{\lambda_{\beta}}
\newcommand{\LLL}{\mathcal{L}_\la}
\newcommand{\LGL}{\mathcal{L}^G_\la}
\newcommand{\QQQ}{\mathcal{Q}_\la}

\def\so{{\mathbb S}^1}
\newcommand{\ddd}{{\mathbb D}_1}

\def\Xint#1{\mathchoice
  {\XXint\displaystyle\textstyle{#1}}%
  {\XXint\textstyle\scriptstyle{#1}}%
  {\XXint\scriptstyle\scriptscriptstyle{#1}}%
  {\XXint\scriptscriptstyle\scriptscriptstyle{#1}}%
  \!\int}
\def\XXint#1#2#3{{\setbox0=\hbox{$#1{#2#3}{\int}$}
    \vcenter{\hbox{$#2#3$}}\kern-.5\wd0}}

\def\dashint{\Xint-}

\newtheorem{thm}{Theorem}[section]
\newtheorem{lem}[thm]{Lemma}
\newtheorem{cor}[thm]{Corollary}
\newtheorem{prop}[thm]{Proposition}
\newtheorem{rem}[thm]{Remark}
\newtheorem{defn}[thm]{Definition}

\newcommand{\be}{\begin{equation}}
\newcommand{\ee}{\end{equation}}
\newcommand{\bea}{\begin{align}}
\newcommand{\eea}{\end{align}}
\newcommand{\beann}{\begin{align*}}
\newcommand{\eeann}{\end{align*}}
\newcommand{\nnn}{\nonumber}

\begin{document}

\baselineskip=18pt

\title{On compound vortices in a two-component \\ Ginzburg--Landau functional}
\author{{\Large Stan Alama}\footnote{Dept. of Mathematics and Statistics,
McMaster Univ., Hamilton, Ontario, Canada L8S 4K1.  Supported
by an NSERC Research Grant.  {\tt alama@mcmaster.ca, bronsard@mcmaster.ca}} \and  {\Large Lia Bronsard${}^*$}
  \and {\Large Petru Mironescu}\footnote{Universit\'e de Lyon;
CNRS;
Universit\'e Lyon 1;
Institut Camille Jordan,
43 blvd du 11 novembre 1918,
F-69622 Villeurbanne-Cedex, France.  {\tt mironescu@math.univ-lyon1.fr}
}}

\maketitle

\begin{abstract}
We study the structure of vortex solutions in a Ginzburg--Landau system for two complex valued order parameters.  We consider the Dirichlet problem in the disk in $\mathbb R^2$ with symmetric, degree-one boundary condition, as well as the associated degree-one entire solutions in all of $\RR^2$.  Each problem has degree-one equivariant solutions with radially symmetric profile vanishing at the origin, of the same form as the unique (complex scalar) Ginzburg--Landau minimizer.  We find that there is a range of parameters for which these equivariant solutions are the unique locally energy minimizing solutions for the coupled system.  Surprisingly, there is also a parameter regime in which the equivariant solutions are unstable, and minimizers must vanish separately in each component of the order parameter.
\end{abstract}

\section{Introduction}

We continue our study of the structure of vortices in two-component Ginzburg--Landau functionals begun in \cite{AB2,ABM}.  Let $\Omega\subset\RR^2$ be a smooth, bounded domain, and
$\Psi\in H^1(\Omega;\CC^2)$.  We define an energy functional,
\begin{equation}\label{EOmega}   
   E_\eps (\Psi;\Omega) = \into \left\{ \frac12|\nabla\Psi|^2 
         + {1\over 4\eps^2} \left(|\Psi|^2-1\right)^2
             + {\beta\over 4\eps^2} \left( |\psi_+|^2 -|\psi_-|^2\right)^2 \right\}\, dx,
\end{equation}
where $\Psi=(\psi_+,\psi_-)$, $\beta>0$ and $\eps>0$ are parameters.
Energy functionals of a  form similar to $E_\eps$ have been introduced in physical models, and at the end of the section we will briefly describe two such contexts:  a Spinor Ginzburg--Landau functional, describing ferromagnetic and antiferromagnetic superconductors, and giving rise to half-integer degree vortices; and a Gross-Pitaevskii functional for a two-component Bose-Einstein condensate.  Although the physical models are more complex, we expect that the essential features of the singular limit $\eps\to 0$ in the physical systems will be well described by the simpler energy \eqref{EOmega}.

As a model problem, we consider \eqref{EOmega} with appropriate Dirichlet boundary conditions, and study the behavior of energy minimizers as $\eps\to 0$.  In the limit, minimizers $\Psi$ should lie on the manifold in $\mathbb{C}^2$ on which the potential $F(\Psi)=\frac14(|\Psi|^2-1)^2 + {\beta\over 4}(|\psi_+|^2-|\psi_-|^2)^2$ vanishes.  That manifold is a 2-torus $\Sigma\subset \mathbb{S}^3\subset\mathbb{C}^2$, parametrized by two real phases  
$$\Psi = \left({1\over\sqrt{2}}e^{i\alpha_+},{1\over\sqrt{2}} e^{i\alpha_-}\right),
$$
and thus a $\Sigma$-valued map $\Psi(x)$ carries a pair of integer-valued degrees around any closed curve $C$, 
$$  \deg(\Psi; C)= [n_+,n_-],\qquad  n_+ = \deg(\psi_+; C) \qquad  n_- = \deg(\psi_-; C).  $$
If the given Dirichlet boundary condition has nonzero degree in either component, then there is no finite energy map $\Psi$ which takes values in $\Sigma$ and satisfies those boundary conditions, and we expect that vortices will be created in the $\eps\to 0$ limit, just as in the classical Ginzburg--Landau model \cite{BBH2}.

An analysis of the global minimizers of the Dirichlet problem, with $\Sigma$-valued boundary data, is given in \cite{AB1,AB2}.  As in the seminal work of Bethuel, Brezis, \& H\'elein \cite{BBH2}, nonzero degree boundary data give rise to vortices in $\Omega$, each of degree one in one (or both) of the two phases $\alpha_\pm$, and the location of the vortices is determined by minimizing a renormalized energy, which is derived by sharp estimates of the interaction energy between the vortices.  The essential difference between the classical Ginzburg--Landau model and the energy \eqref{EOmega} is that there are different species of vortices, allowing for winding in one or both of the two phases, $\alpha_\pm$.  In the renormalized energy expansion, vortices with winding in different components do not interact directly, in the sense that there is no term in the renormalized energy which couples the location of the $\alpha_+$ and $\alpha_-$ vortices.  However, we discovered that there is a very short-range interaction between these two species due to the energy of the vortex cores.  In particular, it is shown (for certain values of the parameter $\beta$,) that it may be beneficial for two vortices of different type (one with degree $[n_+,n_-]=[1,0]$ and one with $[n_+,n_-]=[0,1]$) to coincide in the $\eps\to 0$ limit, rather than to converge to distinct points (as the renormalized energy would normally dictate.)  In this paper, we study the finer structure of these compound vortices:  for $\eps>0$ small we ask, do they resemble Ginzburg--Landau vortices, with $|\Psi|=0$ at a common vortex location; or does each component $\psi_\pm$ vanish separately?

\medskip

To illustrate, we begin with a simple but prototypical example: let $\Omega=\ddd=D(0,1)$, the unit disk, and denote
$$  E_\eps(\Psi)= E_\eps(\Psi; \ddd).  $$
We consider minima (or more generally, critical points) of $E_\eps$ over the space $\HH$, consisting of all functions $\Psi\in H^1(\ddd;\CC^2)$ with the symmetric boundary condition: 
\be\label{sym}   \Psi|_{\partial \ddd} 
              = {1\over\sqrt{2}}\left(e^{i\theta}, e^{i\theta}\right).
\ee
The degree on the boundary is $[n_+,n_-]=[1,1]$, and the energy expansion of \cite{AB2} shows that the minimizers $\Psi_\eps$ converge to a $\Sigma$-valued harmonic map in $\ddd$ with a single limiting vortex at the origin.  Thus, minimizers produce a compound vortex with winding in each phase near the origin.  For $\eps>0$ small but nonzero, do the zeros of the two components coincide or not?  It is easy to verify that if $u_\eps$ minimizes the classical Ginzburg--Landau energy with symmetric boundary condition,
$$  G_\eps(u)=\int_{\ddd} \left[ \frac12 |\nabla u|^2 + {1\over 4\eps^2}(|u|^2-1)^2\right],  \qquad
u|_{\partial \ddd}=e^{i\theta},
$$
then $U_\eps={1\over\sqrt{2}}(u_\eps,u_\eps)$ is a critical point of $E_\eps$ with boundary data \eqref{sym}.  Is the Ginzburg--Landau minimizer $U_\eps$  minimizing for $E_\eps$?

We prove the following:

\begin{thm}\label{thm1}
Let $u_\eps$ minimize $G_\eps$ with $u_\eps|_{\partial \ddd}=e^{i\theta}$ and
$$  U_\eps={1\over\sqrt{2}}(u_\eps(x),u_\eps(x)).  $$
\begin{enumerate}
\item[(i)]  If $\beta\ge 1$, then $U_\eps$ minimizes $E_\eps$ with Dirichlet condition \eqref{sym} for \underbar{every} $\eps>0$.
\item[(ii)]  If $0<\beta<1$, then for all sufficiently small $\eps>0$, $U_\eps$ is \underbar{not} the minimizer of $E_\eps$ with boundary condition \eqref{sym}.
 \end{enumerate}
\end{thm}
The proof of Theorem~\ref{thm1} is based on comparisons between the vortex core energies (see \eqref{cores}), and on our previous results in \cite{AB1,ABM}, and  is the content of Section~\ref{bvpsec}.
For $0<\beta<1$, the form of minimizers is more complex and more interesting.  Our results (see Corollary~\ref{cor5} and Lemma~\ref{lem6},) suggest that the compound $[n_+,n_-]=[1,1]$ vortex  which appears in the $\eps\to 0$ limit actually breaks down into two distinct simple vortices for $\eps>0$.  

\medskip

Although this example seems quite special, in fact the symmetric minimization problem plays an important role in the expansion of the energy in the method of \cite{BBH2}, and we expect that minimizers of this problem accurately describe the structure of vortices near the vortex core.  

Another approach to the core structure of vortices is obtained by blowing up the solution at scale $\eps$ around the vortex center.  After rescaling and passing to the limit, one obtains an entire solution in all of $\RR^2$ to 
the elliptic system,
\begin{equation} \label{eqns} \left.
\begin{gathered}
-\Delta \psi_+ = (1-|\Psi|^2)\psi_+ + \beta (|\psi_-|^2-|\psi_+|^2)\psi_+, \\
-\Delta \psi_- = (1-|\Psi|^2)\psi_- - \beta (|\psi_-|^2-|\psi_+|^2)\psi_-.
\end{gathered} \right\}
\end{equation}
Solutions to (\ref{eqns}) obtained by blowing up minimizers in $\Omega$ will satisfy 
an integrability condition,
\be\label{BMRcond}
\int_{\RR^2} \left\{ \left(|\Psi|^2-1\right)^2 + 
    \beta \left( |\psi_-|^2-|\psi_+|^2\right)^2\right\} dx<\infty,
\ee
analogous to the condition of \cite{BMR} for the classical Ginzburg--Landau equation.  While they have infinite energy measured in the whole of $\RR^2$,  they do inherit a local energy minimizing property, identified by De Giorgi.  For $\Psi\in H^1_{loc}(\RR^2; \mathbb C^2)$ satisfying \eqref{BMRcond}, we denote
\be\label{E}
E(\Psi;\Omega) := E_1(\Psi;\Omega) =
   \int_\Omega \left\{ \frac12 |\nabla\Psi|^2 + \frac14\left(|\Psi|^2-1\right)^2 + {\beta\over 4} \left( |\psi_-|^2 - |\psi_+|^2\right)^2
   \right\}.
\ee

%define a localized energy in any domain $\Omega\subset\RR^2$, 
%\begin{equation}\label{enloc}
% E(\Psi; \Omega) = \int_\Omega \left\{\frac12 |\nabla\Psi|^2
%  + {1\over 4} (|\Psi|^2-1)^2 + {\beta\over 4} (|\psi_+|^2-|\psi_-|^2)^2\right\}.
%\end{equation}

\begin{defn}\label{deGiorgi}
We say that $\Psi$ is a \underbar{locally minimizing}
solution of (\ref{eqns}) if (\ref{BMRcond}) holds 
and if for every bounded  regular domain $\Om\subset\RR^2$,
$$  E(\Psi;\Om)\le E(\Phi;\Om)  $$
holds for every $\Phi=(\varphi_+,\varphi_-)\in H^1(\Om;\CC^2)$
with $\Phi|_{\partial\Om}=\Psi|_{\partial\Om}$.
\end{defn}
For the classical Ginzburg--Landau equation in $\RR^2$, 
$$   -\Delta u = (1-|u|^2)u,  $$
the locally minimizing solutions are completely known.  Combining results by Shafrir \cite{Sh}, Sandier \cite{Sa}, and Mironescu \cite{M2}, the unique nontrivial locally minimizing solution is (up to symmetries) the degree-one equivariant solution, $u=f(r)e^{i\theta}$.

In \cite{ABM} we proved several results on the entire solutions of \eqref{eqns}.  Following the work of \cite{BMR} on the Ginzburg--Landau equations, any solution of \eqref{eqns} satisfying \eqref{BMRcond} has a degree pair at infinity,  (see \cite{ABM}),
$n_\pm = \deg(\psi_\pm;\infty)=\deg\left( {\psi_\pm\over |\psi_\pm|}; S_R\right)$ for all
sufficiently large radii $R$.  We also showed that there exists a unique equivariant solution to \eqref{eqns} for each degree pair $[n_+,n_-]$, but as in the Ginzburg--Landau case we do not expect all those solutions to be local minimizers.  Indeed, it is only for the simplest, ``ground state'' degrees $[n_+,n_-]=[1,0]$ or $[0,1]$ that we can assert the existence of locally minimizing entire solutions.  In \cite{ABM} we show that these vortex solutions  are ``coreless'', that is, $|\Psi|$ is bounded away from zero in $\RR^2$.

For degrees $[1,1]$ at infinity, it is not clear whether or not a locally minimizing solution exists for \eqref{eqns}.  For instance, it is easy to verify that, if $u=f(r)e^{i\theta}$ is the symmetric, degree one solution of the Ginzburg--Landau equations in $\RR^2$, then $U=\left( {1\over\sqrt{2}} u, {1\over\sqrt{2}} u\right)$ solves \eqref{eqns} with degrees $[n_+,n_-]=[1,1]$.  We then ask:  is $U$ a locally minimizing solution, and are there any others?

We prove the following result:
  
\begin{thm}\label{thm2}
\begin{enumerate}
\item[(i)]
For $\beta>1$, $\Psi^*$ is a locally minimizing solution with degree $\deg(\psi^*_\pm,\infty)=1$ if and only if 
\be\label{GLsol}   \Psi^*={1\over\sqrt{2}}
\left( u(x-a)e^{i\phi_+}, u(x-a)e^{i\phi_-}\right),  
\ee
where $\phi_\pm$ are real constants, $a\in \RR^2$ is constant, and $u(x)$ is the (unique) equivariant solution to the Ginzburg--Landau equation in $\RR^2$ with 
$\deg(u,\infty)=1$.
\item[(ii)]  If $0<\beta<1$ and $\Psi^*$ is a locally minimizing solution with degree $\deg(\psi^*_\pm,\infty)=1$, then $|\Psi^*|$ is bounded away from zero in $\RR^2$.  Moreover,
\be\label{S2}
  \int_{\RR^2} \left( |\psi_+|^2 - |\psi_-|^2\right)^2 \ge \pi.  
\ee
\end{enumerate}
\end{thm}
In particular, for $0<\beta<1$, the Ginzburg--Landau solution is {\em not} locally minimizing for \eqref{eqns}.  This implies (see Proposition~\ref{unique}) that for $0<\beta<1$, local minimizers with degree pair $[n_+,n_-]=[1,1]$ must have distinct zeros in each component.  In this way, a locally minimizing solution for $\beta\in (0,1)$ should resemble a gluing together of two simple  vortex solutions (of degrees $[1,0]$ and $[0,1]$), studied in \cite{ABM}.  This has an important implication for the Dirichlet problem:  let $\Omega\subset\RR^2$ be a bounded, smooth domain, and $g:\partial\Omega \to \Sigma$ a given smooth boundary condition.  We then conclude:
\begin{thm}\label{thm4}
Let $0<\beta<1$, and suppose $\Psi_\eps$ minimizes $E_\eps(\Psi;\Omega)$ with Dirichlet boundary condition 
$\Psi_\eps|_{\partial\Omega}=g.$  Then there is some $c>0$ such that, for all $\eps>0$ sufficiently small, $|\Psi_\eps|\ge c$ in $\Omega$.
\end{thm}
In particular, the minimizer in the disk $\Omega=\ddd$ with symmetric boundary condition \eqref{sym} has a single Ginzburg--Landau type vortex (with both components vanishing at the origin) for $\beta\ge 1$, but for $0<\beta<1$ each component vanishes separately, and $|\Psi_\eps|$ is bounded away from zero.  From the analysis of the renormalized energy done in \cite{AB1,AB2}, as $\eps\to 0$ the zeros of $\psi_+$ and $\psi_-$ must tend to the origin.  It is an interesting open question to determine the rate at which they coalesce as $\eps\to 0$.  
If the mutual distance between the zeros in each component is of the order of $\eps$, then blowing up at scale $\eps$ produces a locally minimizing solution to \eqref{eqns} with degree pair $[1,1]$ at infinity.  Necessarily, this local minimizer is non-equivariant, with separated zeros in each component.  On the other hand, if no locally minimizing solution exists with degree pair $[n_+,n_-]=[1,1]$, then  the distance between the two vortices in the boundary-value problem must necessarily be much larger than $\eps$, and the compound vortex breaks down into a distinct pair of $[n_+,n_-]=[1,0]$ and $[n_+,n_-]=[0,1]$ vortices for $\eps>0$. 

The analysis of locally minimizing solutions to \eqref{eqns} is done in sections~\ref{sec3} and \ref{sec4}.  It relies on {\it a priori} estimates of solutions in the spirit of Brezis, Merle, \& Rivi\`ere \cite{BMR} and Shafrir \cite{Sh}.

\medskip

Finally, we return to the symmetric boundary value problem \eqref{sym}.  Recall that Theorem~\ref{thm1} states that for $0<\beta<1$ the symmetric solutions $U_\eps={1\over\sqrt{2}}(u_\eps,u_\eps)$ cannot be local minimizers for small $\eps$.  Using a bifurcation analysis, we provide a more detailed description of how the symmetric solutions lose stability, and the structure of the solutions near the critical value:
\begin{thm}\label{thm3}  Let $0<\beta<1$.  
\begin{enumerate}
\item[(i)]  There exists $\eps_\beta>0$ for which $U_\eps$ is a strict local minimizer of $E_\eps$ for $\eps>\eps_\beta$, and $U_\eps$ is unstable if $\eps<\eps_\beta$.
\item[(ii)] $\eps=\eps_\beta$ is a point of bifurcation for critical points of $E_\eps$ with Dirichlet condition \eqref{sym}.  More precisely, there exists $\delta>0$ and a real-analytic family $\{(\Psi^{t,\xi},\eps(t)\}_{|t|<\delta,\xi\in\so}$ of non-equivariant solutions bifurcating from $U_{\eps_\beta}$ at $t=0$, and these are the only non-equivariant solutions in a $\delta$-neighborhood of $(U_{\eps_\beta}, \eps_\beta)$ in $\HH\times (0,\infty)$.
  Each component of the non-equivariant solutions  $\Psi = (\psi_{+},\psi_{-})$ has exactly one zero $\psi_\pm(z_{\pm})=0,$  and their zeros are antipodal $z_{-}=-z_{+}\neq 0$. 
\end{enumerate}
\end{thm}
 
A more detailed description of the bifurcation from the equivariant solutions $U_\eps$ is given in Theorem~\ref{bifthm}.  Indeed, the analysis of the linearization around $U_\eps$ follows the same steps as for the degree $d\ge 2$ case for the Ginzburg--Landau functional (see Mironescu \cite{M1}), and we show that, apart from the $\mathbb{S}^1$-symmetry of the problem, the ground state eigenspace is simple.  

\medskip

Finally, we briefly discuss two physical contexts for our results on compound vortices.

\subsection*{Fractional degree vortices}

Our original motivation for studying the functional \eqref{EOmega} comes from Spinor Ginzburg--Landau functionals introduced in models of ferromagnetic and antiferromagnetic superconductors \cite{KR} or Bose--Einstein Condensates (BEC) \cite{BEC}.

 Let $\Omega\subset\RR^2$ be a smooth, bounded domain, and
$\Psi\in H^1(\Omega;\CC^2)$.  We define an energy functional,
$$   \mathcal{E}_\eps (\Psi) = \frac12\into \left\{ |\nabla\Psi|^2 
         + {1\over 2\eps^2} \left(|\Psi|^2-1\right)^2
             + {2\beta\over \eps^2} \left( \psi_1\times\psi_2\right)^2 \right\}\, dx,
$$
where $\Psi=(\psi_1,\psi_2)$, $\psi_1\times\psi_2 = \Im(\overline{\psi_1}\psi_2)$,
$\beta>0$ 
and $\eps>0$ are parameters.  The quantity 
$$   S=\psi_1\times\psi_2 = \Im \{ \overline\psi_1 \, \psi_2\}$$ 
is  interpreted as the $z$-component of a spin vector, which in this 
two-dimensional model is assumed to be orthogonal to the plane of $\Omega$.  

As $\eps\to 0$, energy minimizers should converge pointwise to the manifold  on which the 
potential term 
$F(\Psi)= \left(|\Psi|^2-1\right)^2
             + {\beta\over 2} \left( \psi_1\times\psi_2\right)^2$
vanishes.  Since
$\beta>0$, we obtain a two-dimensional surface (a 2-torus)
$\Sigma\subset {\mathbb S}^3\subset \CC^2$ parametrized by two real phases, $\phi,\om$:
$$  \Sigma:\quad \Psi = G(\phi,\om):= (e^{i\phi}\cos\om \, , \, e^{i\phi}\sin\om ). $$  
Notice that $G$ is doubly-periodic with minimal
period $G(\phi+\pi,\om\pm\pi)= G(\phi,\om)$, with each phase 
executing a {\it half} cycle.  For a smooth function $\Psi(x)$ taking 
values in $\Sigma$ and a simple closed curve $C$ contained in the domain of $\Psi$ we may
therefore define a pair of {\em half-integer valued} degrees 
$(d_\phi, d_\om)$ corresponding to the winding numbers of the two phases around
$\Sigma$.  From the above observation, these degrees satisfy
$d_\phi , d_\om \in \frac12 \ZZ,$ and $d_\phi + d_\om \in \ZZ.$  The singularities which appear in energy minimizers as $\eps\to 0$ will thus be half-integer quantized, giving rise to {\em fractional degree vortices.}

The connection between the energies $E_\eps$ and $\mathcal{E}_\eps$ is direct:  by a unitary transformation in the range,
$$  \psi_\pm := {1\over\sqrt{2}}(\psi_1 \pm i\psi_2),  $$
the two are seen to be equal, $\mathcal{E}_\eps(\psi_1,\psi_2)=E_\eps(\psi_+,\psi_-)$.  In these new coordinates, the fractional degree vortex for $(\psi_1,\psi_2)$ with degree pair $(d_\phi,d_\om)$ becomes an integer quantized vortex for $(\psi_+,\psi_-)$, with degree pair 
$[n_+, n_-]=\left[ d_\phi+ d_\omega, d_\phi - d_\omega\right]$.  In particular, the minimal energy fractional degree vortices with $(d_\phi,d_\omega)=(\frac12, \pm\frac12)$ are associated to the integer degrees $[n_+,n_-]=[1,0]$ and $[0,1]$, and the Ginzburg--Landau-like vortex of degree pair $(d_\phi,d_\omega)=(1,0)$ becomes the compound vortex $[n_+,n_-]=[1,1]$ in the new coordinates for $\Psi$.  We note that in the new coordinates, the spin 
$$  S = \psi_1\times\psi_2 = \frac12 (|\psi_-|^2 -|\psi_+|^2)  $$
remains an important quantity.

\subsection*{Two-component BEC}

The functional \eqref{EOmega} may also be derived from the Gross-Pitaevsky energy for a rotating mixture of two BEC, introduced in \cite{Sch}.  In this model, we consider the pair $\Psi=[\psi_+,\psi_-]: \ \RR^2\to \CC^2$ which minimizes
\begin{equation}\label{BEC}
\mathcal{G}_\eps(\Psi) = \int_{\RR^2} \left\{
    \frac12 |\nabla \Psi|^2 - \omega \langle \Psi, i\partial_\theta\Psi\rangle
      + \frac12 V(x)|\Psi|^2
      + {1\over \eps^2}\left[ a_+|\psi_+|^4 + a_- |\psi_-|^4
       + 2b |\psi_+|^2 |\psi_-|^2\right] \right\}, 
\end{equation}
where the constant $\omega$ represents the angular speed of rotation, $V(x)$ the trapping potential, and $a_+,a_->0$ and $b\in\RR$ material constants.  Here, and throughout the paper, we use angle brackets to denote the real scalar product on $\CC$ or $\CC^2$,
$$   \llan  \Phi, \Psi\rran = \Re \left\{ \overline{\Phi}\cdot\Psi\right\}, \quad \Phi,\Psi\in \CC^2;
 \qquad \llan \varphi,\psi\rran = 
 \Re \left\{ \overline{\varphi}\,\psi\right\}, \quad \varphi,\psi\in \CC.  $$
 The energy is to be minimized over the constraints,
$$   \int_{\RR^2} |\psi_+|^2 = m_+, 
\qquad  \int_{\RR^2} |\psi_-|^2 = m_-.
$$
For simplicity we replace the trapping potential $V$ by a ``flat trap'' in the domain $\Omega\Subset\RR^2$ (see \cite{CY} for example), in other words we set $V(x)\equiv 0$ but impose a (Neumann) boundary condition via $\Psi\in H^1(\Omega;\CC^2)$.  We also assume that $a_+=a_-=:a>0$ and assume the two species are balanced,
\begin{equation}\label{L2}
  \dashint_\Omega |\psi_+|^2 = \frac12 = \dashint_\Omega |\psi_-|^2.
\end{equation}
Given the constraints on the $L^2$-norms, we may then complete the square in the quartic terms of the potential by adding in constant multiples of $|\psi_\pm|^2$ without changing the minimizers, to arrive at an energy which more closely resembles \eqref{EOmega},
\begin{equation*}
\tilde{\mathcal{G}}_{\tilde\eps}(\Psi) = \int_{\Omega} \left\{
    \frac12 |\nabla \Psi|^2 - \omega \langle \Psi, i\partial_\theta\Psi\rangle
      + {1\over\tilde\eps^2} (|\Psi|^2 - 1)^2 
       + {\beta\over \tilde\eps^2}\left[ |\psi_-|^2 -|\psi_+|^2\right]^2 \right\}, 
\end{equation*}  
with $\tilde\eps^2={\eps^2\over 2(b+a)}$ and $\beta=-2{(b-a)\over (b+a)}$.  Thus we recover the form of \eqref{EOmega} with $\beta>0$ provided that $-a<b<a$.  As for the single-component BEC energy (see Ignat \& Millot \cite{IM1},) we expect that for $\omega= O(|\ln\tilde\eps|)$, minimizers will have vortices in $\Omega$, and blow-up around vortex centers will result in locally minimizing entire solutions of our system \eqref{eqns}.
In particular, the small-scale structure of vortices in the two-component BEC for $-a<b<a$ will be determined by our analysis of the blow-up problem \eqref{eqns}.  Since the parameter regime $-a<b<0$ corresponds to $\beta>1$, the degree $[n_+,n_-]=[1,1]$ vortices will be radially symmetric, and well-described by the classical Ginzburg--Landau model.  However, the range $0<b<a$ corresponds to $\beta\in (0,1)$, and in this case the $[n_+,n_-]=[1,1]$ vortices will decompose into two separate vortex cores, with $|\Psi|$ bounded away from zero in the core, the ``coreless'' vortices.

\section{The Symmetric Dirichlet Problems}\label{bvpsec}

\setcounter{thm}{0}

We begin with the fundamental boundary value problems, with symmetric data given on the boundary of the unit disk $\Omega=\ddd$:  for $n_\pm\in \{0,1\}$, define

\be\label{Idef}
I_{[n_+,n_-]} (\eps;\beta) =  \min\left\{  E_\eps(\Psi): \ 
     \Psi\in H^1(\ddd;\CC^2), \ \Psi|_{\partial \ddd} 
              = {1\over\sqrt{2}}(e^{in_+\theta}, e^{in_-\theta}) \right\}, 
\ee
For future use, we also define analogous values for the disk ${\mathbb D}_\delta$ centered at the origin and of radius $\delta$:
\begin{align}\label{Jdef}
J_{[n_+,n_-]}(\eps,\delta;\beta)&:= 
 \min\left\{  E_\eps(\Psi): \ 
     \Psi\in H^1({\mathbb D}_\delta;\CC^2), \ \Psi|_{\partial {\mathbb D}_\delta} 
              = {1\over\sqrt{2}}(e^{in_+\theta}, e^{in_-\theta}) \right\} \\
              \nnn
    &= I_{[n_+,n_-]} \left({\eps\over\delta};\beta\right),
\end{align}
by scaling.
For comparison purposes, we also define the analogous quantity for the Ginzburg--Landau functional, 
\begin{align}\label{IGL}
I_{GL}(\eps) & =  \min\left\{  G_\eps(u): \ 
     u\in H^1(\ddd ;\CC), \ u|_{\partial \ddd} 
              =e^{i\theta} \right\}, \\
              \nnn
G_\eps(u) &= \int_{\ddd} \left(  \frac12 |\nabla u|^2 + {1\over 4\eps^2}
     (|u|^2-1)^2\right) dx,    
\end{align}
and $J_{GL}(\eps,\delta)=I_{GL}(\eps/\delta)$ in analogy with (\ref{Jdef}).
Our analysis depends on comparisons between the following {\em vortex core energies,} (see \cite{AB2}, \cite{BBH2}),
\be\label{cores}
\left. \begin{aligned}
  \Qll &= \lim_{\eps\to 0} \left(  \Ill(\eps;\beta) - \pi |\ln\eps| \right),\\
 \Qlo = \Qol &= \lim_{\eps\to 0} \left(  \Ilo(\eps;\beta) - {\pi\over 2} |\ln\eps| \right),
 \\ 
 Q_{GL} =& \lim_{\eps\to 0} \left(  I_{GL}(\eps) - \pi |\ln\eps| \right)
 \end{aligned} \right\}.
\ee
From \cite{AB2} (using the renormalized energy for the problem $\Ill(\eps;\beta)$) we have
\be\label{Q1less}
  \Qll(\beta)\le 2\, \Qlo(\beta), \qquad\mbox{for all $\beta>0$.}
\ee

We will constantly use the following elementary but useful identity:
\bea\nnn
F(\Psi) &:= \left( |\Psi|^2 -1\right)^2 + \beta (|\psi_+|^2-|\psi_-|^2)^2  \\
\label{identity}
&=  2\left[ \left( |\psi_+|^2-\frac12\right)^2 + \left( |\psi_-|^2-\frac12\right)^2 \right]
              + 4(\beta-1)\, S^2,
\end{align}
and we recall that the Spin is given by
$$  S= \frac12 \left( |\psi_-|^2 - |\psi_+|^2\right).  $$

\begin{lem}\label{lem1}
For all $\beta>0$, $\Ill(\eps;\beta)\le I_{GL}(\eps)$.  For $\beta\ge 1$, 
$\Ill(\eps;\beta)= I_{GL}(\eps)$.\\ 
For $\beta>1$  and for any minimizer $\Psi_\eps$ of
$\Ill(\eps;\beta)$, we have
$$  \Psi_\eps(x) = {1\over\sqrt{2}} \left( u(x), u(x) \right),  $$
where  $u(x)$ is a minimizer for the problem $I_{GL}(\eps)$.\\
For $\beta=1$ and for any minimizer $\Psi_\eps$ of
$\Ill(\eps;1)$, we have
$$  \Psi_\eps(x) = {1\over\sqrt{2}} \left( u^1(x) , u^2(x) \right),  $$
where $u^1$, $u^2$ are minimizers for the problem $I_{GL}(\eps)$.
\end{lem}

\begin{proof}
Let $u_\eps$ be the minimizer of $I_{GL}(\eps)$, and 
$U_\eps={1\over\sqrt{2}}(u_\eps,u_\eps)$.  Then, $U_\eps$ is admissible for
$\Ill(\eps;\beta)$ and has spin zero, and therefore
$$  \Ill(\eps;\beta) \le E_\eps(U_\eps) = G_\eps(u_\eps) = I_{GL}(\eps),  $$
for each $\eps>0$ and $\beta>0$.

If $\beta\ge 1$, let $\Psi^\eps=(\psi^\eps_+,\psi^\eps_-)$ minimize $\Ill(\eps;\beta)$, and set $v_{\pm}^\eps=\sqrt{2}\psi^\eps_\pm$.  Using (\ref{identity}), we have, with $S_\eps=\displaystyle\frac{1}{2} (|\psi^\eps_+|^2-|\psi^\eps_-|^2)$,
\begin{align*}
I_{GL}(\eps) &\le \min\{ G_\eps(v^\eps_+) ,G_\eps(v^\eps_-)\} \\
&\le  \frac12\left[  G_\eps(v^\eps_+) + G_\eps(v^\eps_-)\right]  \\
&=  \int_{\ddd} \left\{
     \frac12 |\nabla\Psi^\eps|^2 + {1\over 2\eps^2} 
           \left( \left(|\psi^\eps_+|^2-\frac12\right)^2
                + \left(|\psi^\eps_-|^2-\frac12\right)^2\right) \right\}  \\
  &=  E_\eps (\Psi^\eps) - {\beta-1\over\eps^2} \int_{\ddd} S_\eps^2  \\
  &=  \Ill(\eps;\beta)  - {\beta-1\over\eps^2} \int_{\ddd} S_\eps^2.
\end{align*}
Thus, $\Ill(\eps;\beta)=I_{GL}(\eps)$ for all $\eps>0$ and $\beta\ge 1$.
For $\beta>1$, $S_\eps(x)\equiv 0$, and  $E_\eps(\Psi^\eps)=\frac12 [G_\eps(v^\eps_+) + G_\eps(v^\eps_-)]$, and we conclude that each component must be a minimizer for $I_{GL}(\eps)$. Since 
$|v^\eps_+|=|v^\eps_-|$, the zeros of $v^\eps_\pm$ coincide. By Theorem ~9.1 of Pacard and Rivi\`ere \cite{PR},  $v^\eps_+=v^\eps_-$. 

For $\beta=1$, $E_\eps$ decomposes into a Ginzburg--Landau energy for each component $\psi_\pm$, and the conclusion is immediate.
\end{proof}

For $\beta>1$, there is a net energy saving in replacing two simple vortices, of degree pairs $[0,1]$ and $[1,0]$, by a single compound vortex at the same limiting location:

\begin{lem}\label{lem2}
For all $\beta>1$, $\Qll(\beta)=Q_{GL} < 2\, \Qlo(\beta)$.
\end{lem}

\begin{proof}
The assertion $\Qll(\beta)=Q_{GL}$ for all $\beta\ge 1$ follows trivially from Lemma~\ref{lem1}, and so $\Qlo(\beta)\ge \frac12 Q_{GL}$ for $\beta\ge 1$
follows from  (\ref{Q1less}).
Note also that when $\beta=1$ by (\ref{identity}) the components decouple,
$$   E_\eps(\Psi)= \frac12 \left( G_\eps(\sqrt{2}\psi_+) + G_\eps(\sqrt{2}\psi_-)\right), 
$$
and hence $\Ilo(\eps;1)=\frac12 I_{GL}(\eps)$ for all $\eps>0$.

Suppose that $\Qlo(\beta)=\frac12 Q_{GL}$ for some $\beta> 1$, and let $\Psi^\eps$ attain the minimum $\Ilo(\eps;\beta)$ for that $\beta$.  Then,
\begin{align*}
\Ilo(\eps;\beta) & =  E_{\eps,\beta}(\Psi^\eps)
       = E_{\eps,1}(\Psi^\eps) + {\beta -1\over\eps^2}\int_{\ddd}  S_\eps^2 \\
    &\ge \Ilo(\eps,1) + {\beta -1\over\eps^2}\int_{\ddd}  S_\eps^2  \\
    &\ge \frac12 I_{GL}(\eps) + {\beta -1\over\eps^2}\int_{\ddd}  S_\eps^2.    
\end{align*}
In particular, we conclude that
$$   {1\over \eps^2} \int_{\ddd} S_\eps^2 \to 0.  $$

By the analysis of problem $\Ilo(\eps;\beta)$ in \cite{AB1}, minimizers $\Psi^\eps$ have a vortex ball ${\mathbb D}_\eps$ (of radius $O(\eps)$) with degree 
$\deg(\psi^\eps_+,\partial {\mathbb D}_\eps)=1$, $\deg(\psi^\eps_-,\partial {\mathbb D}_\eps)=0$.
Rescaling by $\eps$ and passing to the limit, these converge to a locally minimizing entire solution
$\Psi^*$ of (\ref{eqns}) in $\RR^2$, with degree $[n_+,n_-]=[1,0]$ at infinity.  The convergence being uniform on any compact set, the limit $\Psi^*$ has spin
$S_*=\frac12 (|\psi^*_-|^2 - |\psi^*_+|^2)\equiv 0$ in $\RR^2$.  This contradicts the main result of \cite{ABM}, where it is proven that $|\psi^*_-|$ is bounded away from zero for such solutions (and hence $S_*>0$ at the zero of $\psi^*_+$.)
\end{proof}

\medskip

We now turn to the case $0<\beta<1$, where the situation is very different.

\begin{lem}\label{lem3}
For all $\beta\in (0,1)$, 
$$   \Qll(\beta)\le 2 \Qlo(\beta)\le Q_{GL} - (1-\beta){\pi\over 4}.  $$
\end{lem}
In particular, for $\beta<1$ we may already conclude that the Ginzburg--Landau solution is {\it not} the minimizer for the problem $\Ill(\eps;\beta)$, and hence Theorem~\ref{thm1} follows from Lemmas~\ref{lem3} and \ref{lem1}.  It is an interesting open question to determine whether the strict inequality $\Qll(\beta)< 2 \Qlo(\beta)$ holds or not.  

\begin{proof}
Let $u_\eps$ minimize $I_{GL}(\eps)$, and set $\psi_+={1\over\sqrt{2}}u_\eps$,
$\psi_-={1\over\sqrt{2}}$.  Then $\Psi=(\psi_+,\psi_-)$ is admissible for
$\Ilo(\eps;\beta)$, and 
\begin{align*}
\Ilo(\eps;\beta) &\le
    E_\eps(\Psi) \\
    &=  \int_{\ddd} \left\{
        \frac14 |\nabla u_\eps|^2 + { (1+\beta)\over 16\eps^2} (1-|u_\eps|^2)^2\right\}\\
    &=  \frac12 G_\eps(u_\eps) - {(1-\beta)\over 16\eps^2} \int_{\ddd} (1-|u_\eps|^2)^2.
\end{align*}
By blow-up and the result of Brezis, Merle, and Rivi\`ere \cite{BMR} we have
$$   {1\over\eps^2}\int_{\ddd} (1-|u_\eps|^2)^2 \to 2\pi,  $$
and therefore we conclude that
$$  \Qlo(\beta) \le \frac12 Q_{GL} - (1-\beta){\pi\over 8},  $$
as claimed.
\end{proof}

\begin{lem}\label{lem4}
For all $\beta\in (0,1)$, 
$$  Q_{GL}\le \Qll(\beta)
      + (1-\beta) \liminf_{\eps\to 0} {1\over\eps^2}\int_{\ddd}  S_\eps^2\, dx,  $$
where $S_\eps$ is the spin associated to any minimizer $\Psi_\eps$ of
$\Ill(\eps;\beta)$.
\end{lem}

\begin{proof}
Let $\Psi^\eps$ minimize $\Ill(\eps;\beta)$ with $0<\beta<1$, and set
$v_\pm = \sqrt{2}\psi^\eps_\pm$.  Each $v_\pm$ is admissible for $I_{GL}(\eps)$, hence
\bea\nnn
I_{GL}(\eps) &\le
\frac12 \left(  G_\eps(v_+) + G_\eps(v_-)\right) \\
\label{l4comp}
 &= \int_{\ddd} \left[  \frac12 |\nabla\Psi^\eps|^2 + {1\over 2\eps^2}
        \left( \left(|\psi^\eps_+|^2 - \frac12\right)^2 + 
         \left(|\psi^\eps_-|^2 - \frac12\right)^2\right)\right]  \\ \nnn
  &= E_\eps (\Psi^\eps) + {(1-\beta)\over \eps^2}\int_{\ddd} S_\eps^2  \\ \nnn
  &= \Ill(\eps;\beta) + {(1-\beta)\over \eps^2}\int_{\ddd} S_\eps^2.
\end{align}
Subtracting $\pi|\ln\eps|$ from both sides and passing to the limit we obtain the desired conclusion.
\end{proof}

By putting together Lemma~\ref{lem3} with Lemma~\ref{lem4} we obtain interesting information about the minimizers of the problem $\Ill(\eps;\beta)$:  when $0<\beta<1$, a fixed amount of the core energy must come from the spin term.

\begin{cor}\label{cor5}
If $\Psi_\eps$ is any minimizer of $\Ill(\eps;\beta)$ with $0<\beta<1$, and $S_\eps$ is its spin, then
$$   \liminf_{\eps\to 0} {1\over\eps^2}\int_{\ddd}  S_\eps^2\, dx \ge {\pi\over 4}.  $$
\end{cor}

A similar calculation applies also to the fractional degree case:

\begin{lem}\label{lem6}
Assume $0<\beta<1$.  Then, for any minimizer $\Psi_\eps$ of $\Ilo(\eps;\beta)$,
its spin $S_\eps$ satisfies:
$$   \liminf_{\eps\to 0} {1\over\eps^2}\int_{\ddd}  S_\eps^2\, dx \ge {\pi\over 8}.
$$
\end{lem}

\Pf
Let $\Psi^\eps$ minimize $\Ilo(\eps;\beta)$ and $v_\pm = \sqrt{2}\psi^\eps_\pm$.
Now only $v_+$ is admissible for the problem $I_{GL}(\eps)$, so
\begin{align*}
\frac12 I_{GL}(\eps) & \le  \frac12 G_\eps(v_+) \\
&\le \frac12 \left[ G_\eps(v_+) + G_\eps(v_-)\right] \\
& \le  E_\eps(\Psi^\eps) + {(1-\beta)\over\eps^2} \int_{\ddd} S_\eps^2\\
&=  \Ilo(\eps;\beta) + {(1-\beta)\over\eps^2} \int_{\ddd} S_\eps^2,
\end{align*}
where we have applied the same reasoning in the next-to-last line as in the computation 
(\ref{l4comp}) above.   Subtracting $\frac{\pi}2 |\ln\eps|$ from both sides and passing to the limit we have
$$  \frac12 Q_{GL} \le \Qlo(\beta) 
    + \liminf_{\eps\to 0} {(1-\beta)\over\eps^2} \int_{\ddd} S_\eps^2.  $$
The conclusion then follows from Lemma~\ref{lem3}.
\QED

\section{A different way to measure core energies}\label{sec3}

\setcounter{thm}{0}

In this section we consider entire solutions $\Psi$ in $\RR^2$, satisfying the integrability condition \eqref{BMRcond}, with given degree $[n_+,n_-]$ at infinity.  Measured on all of $\RR^2$, the energy (defined in \eqref{E}) of such a solution diverges.  However, when properly renormalized, there is a well-defined core energy, defined as the limit below:
\begin{lem}\label{coreenergy}  Let $\Psi$ solve \eqref{eqns} in $\RR^2$, satisfying \eqref{BMRcond}.  Then, the following limit exists:
\be\label{bounded} \lim_{R\to\infty} \left[ E(\Psi; {\mathbb D}_R)
         - {\pi\over 2} (n_+^2+n_-^2)\ln R\right].
\ee  
\end{lem}
For locally minimizing $\Psi$ we expect more.  In the case of the single Ginburg--Landau equation 
\be\label{GLeqn}  -\Delta U = (1-|U|^2)U,
\ee
 it is known from Shafrir \cite{Sh} that the only nontrivial solutions which are locally minimizing for the energy
$$  G_*(U; \Omega) = \int_\Omega \left\{
\frac12 |\nabla  U|^2 + \frac14 (|U|^2-1)^2\right\}, $$
 have degree $\deg(U,\infty)=\pm 1$.  And for those (unique,  by \cite{M2}) local minimizers, the analogous limit \eqref{bounded} coincides with the vortex core energy defined via the symmetric Dirichlet problem \eqref{IGL},
$$  Q_{GL}=\lim_{R\to\infty} \left(  G_*(u; {\mathbb D}_R)
         - \pi\ln R\right).
$$
We will show that the same is true for $\Psi$:

\begin{prop}\label{n01}
A nontrivial local minimizer of \eqref{eqns} satisfying \eqref{BMRcond} must have degrees $n_\pm\in \{0,\pm 1\}$. 
\end{prop}

\begin{prop}\label{prop1}  For any $\beta>0$ and for any given degrees $n_\pm\in \{0,\pm 1\}$, if $\Psi$ is an entire solution satisfying (\ref{BMRcond}), then
$$  \lim_{R\to\infty} \left[  E(\Psi; {\mathbb D}_R)
         - {\pi\over 2} (|n_+| + |n_-|)\ln R\right] \ge Q_{[n_+,n_-]}.  $$
If in addition $\Psi$ is a local minimizer of energy, then equality holds in the above.
\end{prop}

We begin by proving Lemma~\ref{coreenergy}.

\begin{proof}[Proof of Lemma~\ref{coreenergy}]
By the estimates in \cite{ABM}, there exists $R_0>0$ for which the solution  $\Psi(x)$  admits a decomposition for $|x|\ge R_0$ in the following form:
\begin{equation}\label{decomp}\left.
\begin{gathered}
\psi_\pm(x) = \rho_\pm(x)\exp [i\alpha_\pm(x)], \qquad
\alpha_\pm(x)= n_\pm\theta +\chi_\pm(x), \\
\text{with
$\chi_\pm(x)\to \phi_\pm$ (constants) uniformly as $|x|\to\infty$. }
\end{gathered}
\right\}
\end{equation}
Without loss of generality we may take $\phi_\pm=0$, so 
$\chi_\pm(x)\to 0$ uniformly as $|x|\to\infty$. 
Moreover, the estimates in \cite{ABM} imply that for large $r$,
\begin{gather}\label{est1}
\left|\rho_\pm -\rstwo\right| \le {c\over r^2} \\
\label{est2}
|\nabla\rho_\pm(x)|\le {c\over r^3} \\
\label{est3}
\int_{|x|\ge R_0}\left[ |\nabla \rho_\pm|^2 + |\nabla\chi_\pm|^2\right] <\infty,
\end{gather}
for $R_0$ sufficiently large that the decomposition \eqref{decomp} holds.

First, we observe that for any $R>R_0$, by an integration by parts we have:
$$  \int_{{\mathbb D}_R\setminus {\mathbb D}_{R_0}} \nabla \chi_\pm\cdot\nabla\theta =0.  $$
Thus, 
$$  E(\Psi; {\mathbb D}_R\setminus {\mathbb D}_{R_0}) - {\pi\over 2}(n_+^2+ n_-^2)\ln{R\over R_0}
     =  \int_{{\mathbb D}_R\setminus {\mathbb D}_{R_0}} f,  $$
with
\begin{multline*}  f = \sum_\pm \left[ |\nabla \rho_\pm|^2 
+ \left(\rho_\pm^2-{1\over\sqrt{2}}\right) {n_\pm^2\over r^2} + \rho_\pm^2 |\nabla\chi_\pm|^2  +
2n_\pm \left(\rho_\pm-{1\over\sqrt{2}}\right)\nabla\theta\cdot\nabla\chi_\pm]
\right] \\  + 
\frac14 (\rho_+^2 + \rho_-^2-1)^2 + {\beta\over 4} (\rho_+^2-\rho_-^2)^2. 
\end{multline*}
Using the estimates \eqref{est1}--\eqref{est3}, $f$ is integrable in $\RR^2\setminus {\mathbb D}_{R_0}$;  writing 
$$   E(\Psi; {\mathbb D}_R) - {\pi\over 2}(n_+^2+ n_-^2)\ln R =
   E(\Psi;  {\mathbb D}_{R_0}) - {\pi\over 2}(n_+^2+ n_-^2)\ln R_0
     + \int_{{\mathbb D}_R\setminus {\mathbb D}_{R_0}} f,
$$
we conclude that the limit $R\to\infty$ exists.
\end{proof}

Next, we do a patching argument as in \cite{Sh}.

\begin{lem}\label{patch}
Let $\Psi$ be an entire solution of (\ref{eqns}) satisfying (\ref{BMRcond}).
Then, there exists a family $\tilde\Psi_R\in H^1({\mathbb D}_R;\CC^2)$ of functions so that
\begin{align*}
&\tilde\Psi_R(x) = \Psi(x) \qquad \mbox{for $|x|\le {R\over 2}$,} \\
&  \tilde \Psi_R(x) = \rstwo\left[ e^{i(n_+\theta+\phi_+)}, e^{i(n_-\theta+\phi_-)} \right],
\qquad\mbox{on $|x|=R$, for constants $\phi_\pm\in \RR$,} \\
&  
\int_{{\mathbb D}_R} |\nabla\tilde\Psi_R|^2 = \int_{{\mathbb D}_R} |\nabla\Psi|^2 + o(1), \quad
\int_{{\mathbb D}_R} (|\tilde\Psi_R|^2-1)^2 = \int_{{\mathbb D}_R} (|\Psi|^2-1)^2 +o(1), \quad \\
 &
\int_{{\mathbb D}_R} \tilde S_R^2 = \int_{{\mathbb D}_R} S^2 +o(1),
\end{align*}
as $R\to \infty$, where $\tilde S_R=\frac12 (|\psi_{R,-}|^2 - |\psi_{R,+}|^2)$.
In particular, 
$$E(\tilde\Psi_R; {\mathbb D}_R) = E(\Psi; {\mathbb D}_R) + o(1)$$
as $R\to\infty$.
\end{lem}

\Pf
We employ the same decomposition \eqref{decomp} for $\psi_\pm(x)$, 
$|x|\ge R_0$, as in the proof of Lemma~\ref{coreenergy}.  Define the cutoff,
$$   L(r)=L_R(r)=
  \begin{cases} 0 &  \text{if $r\le {R\over 2}$}, \\
     {\ln (2r/R) \over \ln 2}, & \text{if ${R\over 2}\le r\le R$,}  \\
          1 & \text{if $r\ge R$}.
          \end{cases}
$$  
We define our modification
\begin{equation}\label{psitilde}
\left.
\begin{gathered}
\tilde\Psi_R=\left(\tilde\psi_{R,+},\tilde\psi_{R,-}\right), \qquad
\tilde\psi_{R,\pm}=\tilde\rho_\pm(x)\exp[i\tilde\alpha_\pm], \\
  \tilde\rho_\pm(x)= \rstwo L(r) + (1-L(r))\rho_\pm(x), \qquad
   \tilde\alpha_\pm(x)= n_\pm\theta + (1-L(r))\chi_\pm (x). 
   \end{gathered}
 \right\}   
\end{equation}
Then,
\begin{align*}
|\nabla\tilde\rho_\pm|^2 - |\nabla\rho_\pm|^2
&= {1\over (\ln 2)^2 r^2} \left(\rstwo - \rho_\pm\right)^2
              + (L^2-2L) |\nabla \rho_\pm|^2  \\
         &\qquad
             + {2\over r\ln 2} (1-L) \left(\rstwo - \rho_\pm\right)
                    {\partial\rho_\pm\over\partial r}.
\end{align*}
By combining the estimates  $0\le L(x)\le 1$, $\rho_+^2+\rho_-^2<1$, $|\nabla\rho_\pm|\le Cr^{-3}$ (see \cite{ABM}) with (\ref{est3}),
we obtain
$$  \int_{R/2<|x|<R} \left|  |\nabla\tilde\rho_\pm|^2 - |\nabla\rho_\pm|^2 \right| dx
   \to 0  $$
as $R\to \infty$.

Let $C=[\ln 2]^{-1}$.  Then,
\begin{align*}
|\nabla\tilde \alpha_\pm|^2 - |\nabla\alpha_\pm|^2
  &= |\nabla\chi_\pm|^2 [(1-L)^2 -1] + {C\over r^2}\chi_\pm^2
      -2L{n_\pm\over r} (\nabla\chi_\pm\cdot \hat\theta) \\
   &\qquad    - {2C\over r}(1-L)\chi_\pm (\nabla\chi_\pm\cdot \hat r).
\end{align*}
We expand
\bea\nnn
\tilde\rho_\pm^2|\nabla\tilde\alpha_\pm|^2 - 
        \rho_\pm^2|\nabla\alpha_\pm|^2
&= \rho_\pm^2\left(|\nabla\tilde\alpha_\pm|^2-|\nabla\alpha_\pm|^2\right)
\\ \nnn
&\qquad
+|\nabla\alpha_\pm|^2\left(\tilde\rho_\pm^2 - \rho_\pm^2\right)  \\
\label{exp}
& \qquad
+ \left(\tilde\rho_\pm^2 - \rho_\pm^2\right)
   \left(|\nabla\tilde\alpha_\pm|^2-|\nabla\alpha_\pm|^2\right).
\end{align}
Defining $A=\{R/2<|x|<R\}$ and taking each term separately,
\begin{align*}
\left|\int_A 
\rho_\pm^2\left(|\nabla\tilde\alpha_\pm|^2-|\nabla\alpha_\pm|^2\right)\right|
&\le
\int_A\left||\nabla\tilde\alpha_\pm|^2-|\nabla\alpha_\pm|^2\right| \\
&\le  C\int_A \left[ |\nabla\chi_\pm|^2  + {1\over r}|\nabla\chi_\pm|
              + {1\over r^2} \chi_\pm^2\right],
\end{align*}
with constant $C$ independent of $R$.
The first term tends to zero by (\ref{est3}) directly.  For the second,
$$   \int_A {1\over r}|\nabla\chi_\pm|
        \le \left[ 2\pi\int_{R/2}^R {dr\over r} 
               \int_{|x|\ge R} |\nabla\chi_\pm|^2\right]^{1/2}
               \le 
                  \left[2\pi\ln 2 \int_{|x|\ge R} |\nabla\chi_\pm|^2\right]^{1/2}\to 0,
$$
again by (\ref{est3}).  For the last term,
$$  \int_A {1\over r^2}\chi_\pm^2
            \le  2\pi\int_{R/2}^R {dr\over r} \, \sup_{|x|\ge R/2} \chi_\pm^2
                  =2\pi\ln 2\, \sup_{|x|\ge R/2} \chi_\pm^2 \to 0,  $$
since $\chi_\pm\to 0$ uniformly as $|x|\to\infty$.  Thus, the first term of (\ref{exp}) tends to zero as $R\to\infty$.

For the second term of (\ref{exp}), note that by (\ref{est1}),
$$  \int_A |\tilde\rho_\pm^2-\rho_\pm^2|\, |\nabla\alpha_\pm|^2
\le \int_A {c\over r^2}\left( {n_\pm^2\over r^2} + |\nabla\chi_\pm|^2\right)
\to 0, $$
again using (\ref{est3}) for the second piece.
The third integral is estimated in a similar way to the first one and also vanishes as $R\to\infty$.  In conclusion,
$$  \left|\int_A  \tilde\rho_\pm^2|\nabla\tilde\alpha_\pm|^2 - 
              \rho_\pm^2|\nabla\alpha_\pm|^2 \right|\to 0.  $$
The estimate (\ref{est1}) also implies
$$  (|\Psi|^2-1)^2 + 4S^2 = O\left(r^{-4}\right)\quad \text{and}\quad  (|\tilde\Psi_R|^2-1)^2 + 4\tilde S_R^2 = O\left(r^{-4}\right).$$
Therefore, we have
$$  \int_A |(|\Psi|^2 -1)^2 - (|\tilde\Psi_R|^2-1)^2| + |S^2 - \tilde S_R^2| \to 0, $$
as $R\to \infty$.
Putting these results together we obtain
$$  |E(\tilde\Psi_R; {\mathbb D}_R) - E(\Psi; {\mathbb D}_R)| = 
    |E(\tilde\Psi_R; A) - E(\Psi; A)|\to 0$$
    as $R\to\infty$, which completes the proof of the lemma.
\QED

\bigskip

Note that by the same procedure as above, but with the choice
$$   \hat\rho_\pm(x) = \rstwo (1-L_R(r)) + L_R(r)\rho_\pm(x), \qquad
     \hat\alpha_\pm(x) = n_\pm\theta + L_R(r)\chi_\pm,  $$
(and $L_R$ as in the proof of Lemma~\ref{patch},)
 we obtain the opposite patching result, connecting a given solution outside a large ball ${\mathbb D}_R$ to the symmetric boundary condition on $\partial {\mathbb D}_{R/2}$:

\begin{lem}\label{patchtoo}
Let $\Psi$ be an entire solution of (\ref{eqns}) satisfying (\ref{BMRcond}).
Then, there exists a family $\hat\Psi_R\in H^1({\mathbb D}_R\setminus {\mathbb D}_{R/2};\mathbb C^2)$ of functions so that
\begin{align*}
&\hat\Psi_R(x) = \Psi(x) \qquad \mbox{for $|x|=R$,} \\
&  \hat \Psi_R(x) = \rstwo\left( e^{i(n_+\theta+\phi_+)}, e^{i(n_-\theta+\phi_-)} \right),
\qquad\mbox{on $|x|=R/2$, for constants $\phi_\pm\in \RR$,} \\
&  
\int_{{\mathbb D}_R\setminus {\mathbb D}_{R/2}} |\nabla\hat\Psi_R|^2 =
  \int_{{\mathbb D}_R\setminus {\mathbb D}_{R/2}} |\nabla\Psi|^2 + o(1), \\
& \int_{{\mathbb D}_R\setminus {\mathbb D}_{R/2}} (|\hat\Psi_R|^2-1)^2 = 
\int_{{\mathbb D}_R\setminus {\mathbb D}_{R/2}} (|\Psi|^2-1)^2 +o(1), \quad 
\int_{{\mathbb D}_R\setminus {\mathbb D}_{R/2}} \hat S_R^2 = \int_{{\mathbb D}_R} S^2 +o(1),
\end{align*}
as $R\to \infty$, where $\hat S_R=\frac12 (|\psi_{R,-}|^2 - |\psi_{R,+}|^2)$.
In particular, 
$$   E(\hat\Psi_R; {\mathbb D}_R\setminus {\mathbb D}_{R/2}) = 
E(\Psi; {\mathbb D}_R\setminus {\mathbb D}_{R/2}) + o(1)  $$
as $R\to\infty$.
\end{lem}

\begin{proof}[Proof of Proposition~\ref{n01}]
The proof closely follows that of Theorem 2 in \cite{Sh}.  Let $\Psi$ be a local minimizer.  If either $|n_+|\ge 2$ or $|n_-|\ge 2$, we must have $n_+^2 + n_-^2 > |n_+| + |n_-|$, and hence Lemma~\ref{coreenergy} implies that for all $R$ sufficiently large,
\be\label{contradict} 
 \lim_{R\to\infty}{E(\Psi; {\mathbb D}_R)\over \ln R} 
 = {\pi\over 2}(n_+^2 + n_-^2)
 > {\pi\over 2}(|n_+| + |n_-|). 
\ee

By Lemma~\ref{patchtoo}, for $R$ large we obtain $\hat\Psi_R$ with constants $\phi_\pm$, defined in ${\mathbb D}_R\setminus {\mathbb D}_{R/2}$.
Denote by
$$  G_*(U;\Omega):= G_{\eps=1}(U;\Omega)=
\int_\Omega\left\{  \frac12 |\nabla U|^2 + \frac14 (|U|^2-1)^2
\right\},  $$
the Ginzburg--Landau energy for $U\in H^1_{loc}(\Omega;\mathbb C)$.
Taking $\Omega= {\mathbb D}_{R/2}$, let $U_R^\pm$ minimize the Ginzburg--Landau energy $G_*$
 with boundary condition $U_R^\pm|_{\partial {\mathbb D}_{R/2}}=e^{i [n_\pm\theta + \phi_\pm]}$.  
By the results of Brezis, Bethuel, \& H\'elein \cite{BBH2},
$$  G_*(U_R^\pm; {\mathbb D}_{R/2}) = \pi |n_\pm|  \ln (R/2) + O(1).  $$
Now let
$$  \Phi_R(x) = \begin{cases}
         \left(\rstwo U_R^+, \rstwo U_R^-\right), &\text{in ${\mathbb D}_{R/2}$}\\
         \hat\Psi_R(x), &\text{in ${\mathbb D}_R\setminus {\mathbb D}_{R/2}$}
         \end{cases}.
$$
Since $\Psi$ is a local minimizer, we have
\begin{align*}
E(\Psi; {\mathbb D}_R) &\le E(\Phi_R; {\mathbb D}_R) \\
 &= \frac12\left[ G_*(U_R^+;{\mathbb D}_{R/2}) +  G_*(U_R^-;{\mathbb D}_{R/2})\right] 
    + E(\hat\Psi_R; {\mathbb D}_R\setminus {\mathbb D}_{R/2})  \\
  & = \frac{\pi}2\left( |n_+| + |n_-|\right) \ln (R/2)  
    + E(\Psi; {\mathbb D}_R\setminus {\mathbb D}_{R/2}) + O(1).          
\end{align*}
From Lemma~\ref{coreenergy}, it follows that
$$  E(\Psi; {\mathbb D}_R\setminus {\mathbb D}_{R/2}) =\frac\pi 2(n_+^2+n_-^2)\ln 2 + o(1),  $$
as $R\to \infty$, and hence
$$  E(\Psi; {\mathbb D}_R) \le {\pi\over 2}(|n_+| + |n_-|)\ln R + O(1).  $$
This contradicts \eqref{contradict}.
\end{proof}

\medskip

\begin{proof}[Proof of Proposition~\ref{prop1}]
Let $\Psi$ be an entire solution to (\ref{eqns}) satisfying (\ref{BMRcond})
with degrees $n_\pm\in \{0,1\}$ at infinity.  (The degree -1 case may be obtained by complex conjugation.)  Multiplying each component $\psi_\pm$ by a complex constant of modulus one if necessary, we may assume that
$\psi_\pm\to\rstwo e^{in_\pm\theta}$ as $|x|\to \infty$ (that is, $\phi_\pm=0$.)  For  large $R$, let $\tilde \Psi_R$ be as in
Lemma~\ref{patch}, so that
$\tilde\Psi_R|_{\partial {\mathbb D}_R}=\rstwo (e^{in_+\theta},e^{in_-\theta})$.
Note that by scaling,
\be\label{scaling}  J_{[n_+,n_-]}(1,R;\beta)= 
\min\left\{E(\Psi; {\mathbb D}_R): \ \Psi|_{\partial {\mathbb D}_R}
        =\rstwo (e^{in_+\theta},e^{in_-\theta})\right\} 
          = I_{[n_+,n_-]}\left({1\over R};\beta\right).
\ee
Using \eqref{IGL} we conclude that
\begin{align*}
E(\Psi; {\mathbb D}_R) &= E(\tilde \Psi_R; {\mathbb D}_R) - o(1) \\
&\ge  J_{[n_+,n_-]}(1,R;\beta) - o(1)
  = {\pi\over 2}(n_++n_-)\ln R + Q_{[n_+,n_-]} - o(1),
\end{align*}
which proves the first assertion in Proposition~\ref{prop1}.

\smallskip

Now assume in addition that $\Psi$ is a local minimizer.  We now employ Lemma~\ref{patchtoo}, with 
\begin{equation}\label{psihat}
\left.
\begin{gathered}
\hat\Psi_R=\left(\hat\psi_{R,+},\hat\psi_{R,-}\right), \qquad
\hat\psi_{R,\pm}=\hat\rho_\pm(x)\exp[i\hat\alpha_\pm], \\
  \hat\rho_\pm(x)= \rstwo (1-L_R(r)) + L_R(r)\rho_\pm(x), \qquad
   \hat\alpha_\pm(x)= n_\pm\theta + L(r)\chi_\pm (x)
   \end{gathered}
 \right\}   ,
\end{equation}
so that
$$   \hat\Psi_R|_{\partial {\mathbb D}_{R/2}}= \rstwo(e^{in_+\theta}, e^{in_-\theta}), \qquad
      \hat\Psi_R|_{\partial {\mathbb D}_{R}}=\Psi |_{\partial {\mathbb D}_{R}},  
$$
and for all $R>R_0$,
$$  E(\hat\Psi_R; {\mathbb D}_{R}\setminus {\mathbb D}_{R/2}) = 
E(\Psi; {\mathbb D}_{R}\setminus {\mathbb D}_{R/2}) + o(1) = 
{\pi\over 2}(n_++n_-)\ln 2 + o(1).
$$
Next, we let $V_R(x)$ be a minimizer for the problem $I_{[n_+,n_-]}(1,R/2;\beta)$ with symmetric data.  By scaling,
$$  E(V_R; {\mathbb D}_{R/2}) = J_{[n_+,n_-]}(1,R/2;\beta) 
= I_{[n_+,n_-]}(2/R;\beta)= {\pi\over 2}(n_+ +n_-)\ln \left({R\over 2}\right) + Q_{[n_+,n_-]} + o(1).  $$

We now define a trial function in ${\mathbb D}_{R}$,
$$  \Phi_R(x)= \begin{cases}
   \hat\Psi_R(x), &\text{ if  $R/2\le |x|\le R$} \\
   V_R(x), & \text{if $|x|\le R/2$}\end{cases}.
$$
Since $\Phi_R=\Psi$ on $\partial {\mathbb D}_{R}$, and $\Psi$ is a local minimizer, we have:
\begin{align*}
E(\Psi; {\mathbb D}_{R}) &\le E(\Phi_R; {\mathbb D}_{R})  \\
&= E(V_R; {\mathbb D}_{R/2}) + E(\hat\Psi_R; {\mathbb D}_R\setminus {\mathbb D}_{R/2}) \\
&= {\pi\over 2}(n_++n_-)\ln R + Q_{[n_+,n_-]} + o(1).
\end{align*}
\end{proof}

\section{Locally minimizing solutions}\label{sec4}

\setcounter{thm}{0}

We are ready to prove Theorem~\ref{thm2} stated in the Introduction, concerning locally minimizing solutions of the system \eqref{eqns}.

\begin{proof}[Proof of Theorem~\ref{thm2}]
First, if $\Psi^*$ has the form 
$$ \Psi^*={1\over\sqrt{2}}
\left( u(x-a)e^{i\phi_+}, u(x-a)e^{i\phi_-}\right), 
$$
 then it is a local minimizer.  By Theorem~\ref{thm1}, the minimizer for the boundary-value problem $\Ill(\eps;\beta)$ coincides with ($\sqrt{2}$ times) the Ginzburg-Landau minimizer in each component.  By blow-up around the vortex we converge to a solution to the system in $\RR^2$ which is a local minimizer.  Since each component of the minimizer in $\ddd$ is a solution of the Ginzburg--Landau equation, the solution we obtain is of the form above, and hence such solutions are local  minimizers.

\smallskip

Now, assume $\Psi^*$ is any local  minimizer with degree $\deg(\psi^*_\pm, \infty)=1$.
We claim that its spin $S_*=\frac12(|\psi^*_-|^2-|\psi^*_+|^2)=0$.  Suppose not, and let
$$  \sigma = \int_{\RR^2} S_*^2 \, dx > 0.  $$
Choose $R_0$ large enough so that $\int_{{\mathbb D}_{R_0}} S_*^2 \, dx \ge \sigma/2$.

We use Lemma~\ref{patch} with $n_\pm=1$
to bound the energies $E(\Psi^*; {\mathbb D}_R)$ from below.  
Let $\tilde \Psi_R=[\tilde\psi_{R,+},\tilde\psi_{R,-}]$ be as in Lemma~\ref{patch}, 
$u_{R,\pm}=\stwo\tilde\psi_{R,\pm}$, and apply the identity (\ref{identity}):
\bea\nnn
E(\Psi^*; {\mathbb D}_R) &= E(\tilde\Psi_R; {\mathbb D}_R) + o(1) \\
\nnn
  &= \frac12\left(  G_*(u_{R,+}; {\mathbb D}_R) + G_*(u_{R,-}; {\mathbb D}_R)\right) +
       (\beta-1)\int_{{\mathbb D}_R} \tilde S_R^2 +o(1) \\
     \nnn
  &\ge I_{GL}\left({1\over R}\right) + (\beta-1) {\sigma\over 2} + o(1) \\
    \label{upper}
  &= \pi\ln R + Q_{GL} + (\beta-1) {\sigma\over 2} + o(1),
\end{align}
since $u_{R,\pm}=e^{i\theta}$ on $\partial {\mathbb D}_R$, and by scaling we have
$$   J_{GL}(1,R):=\min\{ G_*(u; {\mathbb D}_R): \ u\in H_{e^{i\theta}}^1({\mathbb D}_R)\} = I_{GL}(1/R).  $$

Now let $\hat\Psi_R$ be as in Lemma~\ref{patchtoo}, and define
$$  \Phi_R(x) = \begin{cases}
      \hat\Psi_R(x), & \text{if $R/2\le |x|\le R$}, \\
      \rstwo\left( u_{R/2}e^{i\phi_+}, u_{R/2}e^{i\phi_-}\right), 
        & \text{if $|x|<R/2$},\end{cases}
$$
where $u_{R/2}$ is the minimizer of the Ginzburg--Landau functional in ${\mathbb D}_{R/2}$ with symmetric data, 
$$  G_*(u_{R/2}; {\mathbb D}_{R/2}) = \min\{ G_* (v;{\mathbb D}_{R/2}): \
     v\in H^1_{e^{i\theta}}({\mathbb D}_{R/2})\} = I_{GL}(2/R).  
$$
By the estimates of Lemma~\ref{patchtoo},
\begin{align*}
E(\Phi_R; {\mathbb D}_R) & =   E(\hat\Psi_R; {\mathbb D}_R\setminus {\mathbb D}_{R/2}) + E(\Phi_R; {\mathbb D}_{R/2})  
    \\
  &=  E(\Psi^*; {\mathbb D}_R\setminus {\mathbb D}_{R/2}) + I_{GL}(2/R) + o(1) \\
  &=  \left(  \pi\ln\left({R\over R/2}\right) + o(1)\right)
             + \left(\pi\ln (R/2) + Q_{GL}\right) +o(1) \\
    &= \pi\ln R + Q_{GL} + o(1),
\end{align*}
as $R\to\infty$.  Hence, comparing with (\ref{upper}) we have
$$  E(\Phi_R; {\mathbb D}_R) \le E(\Psi^*; {\mathbb D}_R) - (\beta-1){\sigma\over 2} + o(1),  $$
and for $\sigma>0$ we have that $\Psi^*$ cannot be a local minimizer for $\beta>1$.  Therefore the claim is established, $S^*\equiv 0$ for any local minimizer.

We then conclude that for any local minimizer, $|\psi_-|=|\psi_+|$.  So $\psi_\pm$ have their zeros coincident, and each solves
$$   -\Delta \psi_\pm = (1-2|\psi_\pm|^2)\psi_\pm  $$
in $\RR^2$ with degree one at infinity.  By Mironescu \cite{M2} each is ($\stwo$ times) an equivariant entire solution of the Ginzburg--Landau equation, and therefore $\Psi$ must have the form (\ref{GLsol}).
This completes the proof of (i) of Theorem~\ref{thm2}.

\bigskip

Next we turn to the case $0<\beta<1$.  We first show that if $\Psi$ is a local minimizing entire solution in $\RR^2$ with $n_\pm=1$, then
$$    \int_{\RR^2} S^2\, dx \ge {\pi\over 4}, $$
with $S=\frac12(|\psi_-|^2-|\psi_+|^2)$.
Let
$$  \sigma_R:= \int_{{\mathbb D}_R} S^2\, dx, \qquad \sigma=\int_{\RR^2} S^2\, dx.  $$
First, we bound $E(\Psi; {\mathbb D}_R)$ from below.  Apply Lemma~\ref{patch} to obtain for each $R>0$ a corresponding $\tilde \Psi_R$, and
set $\tilde u_\pm=\tilde\psi_{R,\pm}$.  By Lemma~\ref{patch} and (\ref{identity}),
\begin{align}\nnn
E(\Psi; {\mathbb D}_R) &= E(\tilde \Psi_R; {\mathbb D}_R) - o(1) \\
\nnn
&= \frac12(  G_*(\tilde u_+; {\mathbb D}_R) + G_*(\tilde u_-; {\mathbb D}_R) ) - (1-\beta) \int_{{\mathbb D}_R} \tilde S^2 - o(1) \\
\nnn
&\ge  I_{GL}(1/R) -(1-\beta) \sigma_R - o(1)  \\
\label{lowerbd}
&= \pi\ln R + Q_{GL} - (1-\beta)\sigma - o(1).
\end{align}

For the upper bound we use the local minimality of $\Psi$.  First, as in \cite{BBH2} (see \cite{AB1}) for any pair of points $a_1,a_2\in {\mathbb D}_1$, $\rho>0$ fixed, and all small $\eps>0$
 there exists $\hat\Phi$ with $\hat\Phi|_{\partial {\mathbb D}_1}=\rstwo[e^{i\theta},e^{i\theta}]$
 so that
\begin{align*}  E_\eps(\hat\Phi_\eps)&= \pi\ln\left({1\over \rho}\right)
    + W(a_1,a_2; [1,0], [0,1]) + \Ilo(\eps;\rho) + \Iol(\eps;\rho) + O(\rho)\\
     &= \pi\ln\left({1\over \rho}\right)
    -{\pi\over 2}\left[\ln(1-|a_1|^2) + \ln(1-|a_2|^2)\right] \\
&\qquad
     + J_{[1,0]}(\eps,\rho;\beta) + J_{[0,1]}(\eps,\rho;\beta) + O(\rho),
\end{align*}
where $W$ is the renormalized energy associated to the Dirichlet problem (see \cite{AB1}.)  For any $\eta>0$, fix $a_1=-a_2\neq 0$ close enough to the origin 
and $\rho>0$ such that
$$
E_\eps(\hat\Phi_\eps)\le \pi\ln\left({1\over \rho}\right) + 
  J_{[1,0]}(\eps,\rho;\beta) + J_{[0,1]}(\eps,\rho;\beta) +\eta.
$$
Using scaling and Lemma~\ref{lem3},
\begin{align}\nnn
E_\eps(\hat\Phi_\eps)&\le \pi\ln\left({1\over \rho}\right) + J_{[1,0]}(\eps,\rho;\beta) + J_{[0,1]}(\eps,\rho;\beta) +\eta  \\  \nnn
&= \pi\ln\left({1\over \rho}\right) + 2\left[{\pi\over 2}\ln\left({\rho\over\eps}\right)
   + \Qlo \right] + \eta + o(1) \\
   \label{u1}
&\le   \pi|\ln\eps| + Q_{GL} - (1-\beta){\pi\over 4} + \eta + o(1),
\end{align}
where $o(1)$ is with respect to $\eps\to 0$.

We now construct a function $\Phi_R$ with the same boundary values as $\Psi$ on $\partial {\mathbb D}_R$.  Applying Lemma~\ref{patchtoo}, we set $\Phi_R=\hat\Psi_R$ in 
${\mathbb D}_R\setminus {\mathbb D}_{R/2}$, with 
$\Phi_R|_{\partial {\mathbb D}_{R/2}}=\rstwo(e^{i\theta+\phi_+}, e^{i\theta+\phi_-})$, 
$\phi_\pm$ constants.  Inside ${\mathbb D}_{R/2}$ we define $\eps=2/R$ and 
$\Phi_R=\hat\Phi_{2/R} (2x/R)$, with $\Phi_\eps$ as constructed above.
By rescaling and Lemma~\ref{patchtoo} we have (as $R\to\infty$),
\begin{align*}
  E(\Phi_R; {\mathbb D}_R) &= E(\hat\Psi_R; {\mathbb D}_R\setminus {\mathbb D}_{R/2}) 
       + E(\Phi_R; {\mathbb D}_{R/2}) + o(1) \\
     &= \pi\ln\left({R\over R/2}\right) + E_{{2\over R}}(\hat\Phi_{2/R}) + o(1) \\
     &\le \pi\ln R +
       Q_{GL} - (1-\beta){\pi\over 4} + \eta + o(1).
\end{align*}
Since $\Psi$ is a local minimizer, we obtain the upper bound,
\be\label{ub}
E(\Psi; {\mathbb D}_R) \le \pi\ln R +
       Q_{GL} - (1-\beta){\pi\over 4} + \eta + o(1),
\ee
for any $\eta>0$, in the limit $R\to\infty$.

Putting together (\ref{lowerbd}) and (\ref{ub}), since $\beta<1$ we have
$$  \sigma \ge {\pi\over 4} -{\eta\over 1-\beta}.  $$
Since $\eta>0$ is arbitrary we obtain the estimate \eqref{S2} in Theorem~\ref{thm2}.  As a corollary, we observe that the equivariant, Ginzburg--Landau-like solutions are {\em not} local minimizers for $0<\beta<1$.

To complete the proof of Theorem~\ref{thm2} we require the following uniqueness result for entire solutions with a common zero in each component:

\begin{prop}\label{unique}
Suppose $\Psi$ is an entire solution of \eqref{eqns}, satisfying \eqref{BMRcond}, with $[n_+,n_-]=[1,1]$.  If in addition $\psi_+(x_0)=\psi_-(x_0)=0$ for some $x_0\in\RR^2$, then there exist constants $\phi_\pm\in\RR$ with
\begin{equation}\label{GLlike}  \psi_+(x) = {1\over\sqrt{2}} u(x-x_0) e^{i\phi_+},\qquad
 \psi_- (x) = {1\over\sqrt{2}} u(x-x_0) e^{i\phi_-},  
 \end{equation}
where $u=f(r) e^{i\theta}$ is the equivariant, degree one entire solution to the Ginzburg--Landau equation in $\RR^2$.
\end{prop}

\begin{rem} \rm
The proof of Proposition~\ref{unique} is based on the method of Mironescu \cite{M2}, and is deferred to the end of this section.  The existence or non-existence of a locally minimizing entire solution for \eqref{eqns} with degrees $n_+=1=n_-$ is an interesting open question.  If such a local minimizer were to exist, Proposition~\ref{unique} implies that its vortices must be spatially separated.
\end{rem}

We may now complete the proof of Theorem~\ref{thm2}.  Since $\sigma\ge {\pi\over 4}$, the local minimizer $\Psi$ is not of the form \eqref{GLlike}, and hence by Proposition~\ref{unique},  we conclude that $|\Psi(x)|^2\neq 0$ on $\RR^2$.  By \eqref{est1} there exists $R>0$ so that $|\Psi(x)|^2 \ge \frac14$ for $|x|>R$, while on the compact set ${\mathbb D}_R$,  $|\Psi(x)|^2$ is continuous and non-vanishing, thus bounded away from zero.  This completes the proof of Theorem~\ref{thm2}.
\end{proof}

\bigskip

\begin{rem}\label{betaone} \rm When $\beta=1$ the equations decouple completely, and each component satisfies the Ginzburg--Landau equation in $\RR^2$.  In this special case there is a huge degeneracy, and the solution space with degree $[n+_,n_-]=[1,1]$ is completely described by:
$$  U(x)= U_{a,b,\phi_+\phi_-} = \rstwo\left( u(x-a)e^{i\phi_+}, u(x-b)e^{i\phi_-}\right),  $$
where $u=f(r)e^{i\theta}$ is the unique equivariant solution with degree one, $a,b\in \RR^2$ are arbitrary points in the plane, and $\phi_\pm$ are arbitrary constants.  Since the equivariant solution is a local minimizer for Ginzburg--Landau and because the energy decouples when $\beta=1$, 
$$ E(U;{\mathbb D}_R)=\frac12(G_*(u(x-a); {\mathbb D}_R) + G_*(u(x-b);{\mathbb D}_R)),  $$
it is clear that {\it each} of these solutions is a local minimizer.  The question then arises whether there exist local minimizers with $\beta<1$ which bifurcate from this degenerate family as $\beta\to 1$.  Alas, if we calculate the $L^2$ norm of the spin $S_{a,b,\phi_+,\phi_-}$ associated
to $U_{a,b,\phi_+,\phi_-}$, we have
$$  \int_{\RR^2} S^2_{a,b,\phi_+,\phi_-}\, dx < {\pi\over 4},  $$
and the value approaches $\pi/4$ as $|a-b|\to\infty$.  In particular, by Theorem~\ref{thm2}, no family $\Psi_\beta$ of local minimizers for $\beta\to 1^-$
can converge to a member of the family $U_{a,b,\phi_+,\phi_-}$, and so there is no bifurcation.
\end{rem}

\bigskip

Next, we complete the proof of Theorem~\ref{thm4} stated in the Introduction:

\begin{proof}[Proof of Theorem~\ref{thm4}]
Let $\Psi_\eps\in H^1_g(\Omega;\mathbb C^2)$ be minimizers of $E_{\eps_n,\beta}$ with $0<\beta<1$.  We argue by contradiction, and assume that there is no such constant $c>0$ for which $|\Psi|\ge c$ in $\Omega$.  In that case, there exists a sequence $\eps_n\to 0$ and points $p_n\in \Omega$ for which $\Psi_{\eps_n}(p_n)\to 0$.  By the methods of \cite{BBH2} extended to the energy $E_\eps$ (see \cite{AB2},) $|\Psi_n|$ is bounded away from zero in a neighborhood of the boundary $\partial\Omega$, so  in particular, we can assume dist$(p_n,\partial\Omega)\gg\eps_n$ for $\eps_n$ sufficiently small.  Blowing up around $p_n$ at scale $\eps_n$, by standard estimates a subsequence of the rescaled minimizers converge in $C^k_{loc}$ to a locally minimizing solution $\tilde\Psi$ to \eqref{eqns}, satisfying \eqref{BMRcond}, with $\Psi(0)=0$.  By Proposition~\ref{n01}, the degree pair $\deg(\tilde\Psi,\infty)=[n_+,n_-]$ has one of the forms $[0,0]$, $[\pm 1,0]$, $[0,\pm 1]$, $[\pm 1,\pm 1]$.  We claim that each is impossible.  Indeed, from \eqref{BMRcond}, the only $[0,0]$ solution is constant (of modulus $1/\sqrt{2}$) in each component.  In \cite{ABM} it is proven that the $[0,\pm 1]$ and $[\pm 1, 0]$ locally minimizing solutions never vanish for any $\beta>0$.  And Theorem~\ref{thm2} (ii) asserts the same conclusion for the $[\pm 1,\pm 1]$ local minimizer when $0<\beta<1$.  In conclusion, there must exist $c>0$ for which $|\Psi_\eps|\ge c>0$ as claimed.
\end{proof}

%A similar argument leads to the following improvement of Theorem~\ref{thm4}
%\begin{thm}
%Let $0<\beta<1$, and suppose $\Psi_\eps$ minimizes $E_\eps(\Psi;\Omega)$ with Dirichlet boundary condition 
%$\Psi_\eps|_{\partial\Omega}=g.$  Then there is some $c>0$ such that, for all $\eps>0$ sufficiently small, $|\Psi_\eps|\ge c$ in $\Omega$.
%\end{thm}
%

\bigskip

We conclude the section with the deferred proof of Proposition~\ref{unique}:

\begin{proof}[Proof of Proposition~\ref{unique}]
We follow the uniqueness proof of Mironescu \cite{M2}.  Denote by
$\rho={1\over\sqrt{2}} f$.  Without loss of generality, assume $x_0=0$, and form the quotients
$$   w_\pm(x) = {\psi_\pm(x)\over \rho(|x|)} \in C^\infty(\RR^2\setminus\{0\}).  $$
We may then derive the system of equations satisfied by $w_\pm$, 
\begin{equation}\label{quo}
\left. \begin{gathered}
-\Delta w_+ - 2{\rho'\over \rho} {\partial w_+\over\partial r}
 - {1\over r^2} w_+ = \rho^2 \left[
           2 - |w_+|^2 - |w_-|^2 + \beta(|w_-|^2-|w_+|^2)\right] w_+ \\
  -\Delta w_- - 2{\rho'\over \rho} {\partial w_-\over\partial r}
 - {1\over r^2} w_- = \rho^2 \left[
           2 - |w_+|^2 - |w_-|^2 - \beta(|w_-|^2-|w_+|^2)\right] w_+ .        
\end{gathered} \right\}
\end{equation}
As in the derivation of the Pohozaev identity, we multiply each equation by $x\cdot\nabla w_\pm$, and integrate over the domain ${\mathbb D}_R\setminus {\mathbb D}_{R_0}$.  The resulting identity has the form
$$   F(R)- F(R_0)= G(R,R_0), $$
with
\begin{multline*}  F(R) = \int_{\partial {\mathbb D}_R}
    {R\rho^2\over 4} \left[ (2-|w_+|^2 - |w_-|^2)^2
           + \beta (|w_-|^2-|w_+|^2)^2\right]
           \\ - \sum_\pm \int_{\partial {\mathbb D}_R} \left[
           {1\over 2R} |w_\pm|^2 
            + {R\over 2} \left[
            |\partial_r w_\pm|^2 - |\partial_\tau w_\pm|^2\right] \right].
\end{multline*}
and
\begin{multline*}
  G(R,R_0) = \int_{{\mathbb D}_R\setminus {\mathbb D}_{R_0}}
        \left[ {2r\rho' \over \rho} \left(
            |\partial_r w_+|^2 + |\partial_r w_-|^2 \right)\right.  \\
            \left.
            + {2\rho^2 + 2r\rho\rho'\over 4}
              \left[ (2-|w_+|^2 - |w_-|^2)^2
           + \beta (|w_-|^2-|w_+|^2)^2\right] \right].
\end{multline*}
Using the estimates \eqref{est1}--\eqref{est3} on the solutions $\psi_\pm$ for large $R$, the corresponding facts for the Ginzburg--Landau profile $\rho$, we obtain 
$F(R)-F(R_0)\to 0$ as both $R\to\infty$ and $R_0\to 0$, and thus
conclude that $G(\infty,0)=0$.  Since $\rho>0$ and $\rho'>0$, we
may conclude that $|w_+|=|w_-|=1$ and $\partial_r w_\pm=0$ in 
all of $\RR^2$, and hence $w_\pm=w_\pm(\theta)=e^{i\chi_\pm(\theta)}$, that is  $\psi_\pm = \rho w_\pm = \rho e^{i\chi_\pm(\theta)}$.
  By the estimates in \cite{ABM}, $\psi_\pm - {1\over\sqrt{2}}e^{i(\theta + \phi_\pm)}\to 0$ uniformly as $|x|\to\infty$, with constants
  $\phi_\pm\in\RR$.  In particular, $\chi_\pm=\phi_\pm$ are constant,
  and $\psi_\pm = \rho e^{i(\theta+\chi_\pm(\theta))} = {1\over\sqrt{2}}f(r)e^{i(\theta + \phi_\pm)}$ as claimed.
\end{proof}

\section{Bifurcation of symmetric vortices}\label{bifsec}

\setcounter{thm}{0}

In this section we study the stability and bifurcation of the equivariant solutions of the Dirichlet problem in the unit disk ${\mathbb D_1}$.  For convenience, we replace the usual parameter $\eps$ by $\lambda=\eps^{-2}$ in both $E_\eps$ and the Ginzburg--Landau energy $G_\eps$, and (with abuse of notation) write 
\begin{gather*}
E_\lambda (\Psi) = \int_{{\mathbb D_1}} \left\{ \frac12|\nabla\Psi|^2 
         + {\lambda\over 4} \left(|\Psi|^2-1\right)^2
             + {\lambda\beta\over 4} \left( |\psi_+|^2 -|\psi_-|^2\right)^2 \right\}\, dx, \\
    G_\lambda(u)=\int_{{\mathbb D_1}} \left[ \frac12 |\nabla u|^2 
    + {\lambda\over 4}(|u|^2-1)^2\right].
\end{gather*}
We consider critical points 
$\Psi\in \HH:=\{\Psi\in H^1({\mathbb D_1};\mathbb C^2): \ \Psi|_{\partial {\mathbb D_1}}
=\rstwo (e^{i\theta},e^{i\theta})\}$, which solve the Dirichlet problem,
\be\label{DP}
\left\{
\begin{tabular}{lll}
$-\Delta \psi_+$ &$= \lambda (1-|\Psi|^2)\psi_+  
   +\lambda\beta(|\psi_-|^2-|\psi_+|^2)\psi_+$ &$\text{in ${\mathbb D_1}$}$ \\
$-\Delta \psi_- $&$= \lambda (1-|\Psi|^2)\psi_-  
  -\lambda\beta(|\psi_-|^2-|\psi_+|^2)\psi_-$
 &$\text{in ${\mathbb D_1}$}$\\
\hskip 7mm$\psi_\pm$ &$= \displaystyle\frac{1}{\sqrt{2}} e^{i\theta} $&$\text{on $\partial {\mathbb D_1}$}$
\end{tabular} 
\right. ,
\ee
in the unit disk ${\mathbb D_1}$, for fixed $\beta$, as $\lambda$ ranges in the half-line $\lambda\in (0,\infty)$.  We will show that if $0<\beta<1$, the symmetric vortex solution
$$    U_\lambda=\left({1\over\sqrt{2}} u_\lambda, {1\over\sqrt{2}} u_\lambda\right),  $$
where $u_\lambda=f_\lambda(r)e^{i\theta}$ is the degree-one equivariant solution of the classical Ginzburg--Landau model in the unit disk,
is stable (a strict minimizer of energy) for $\lambda=\eps^{-2}$ small, but loses stability at some $\lambda_\beta=\eps^{-2}_\beta>0$.  

First, we note that for $\lambda$ small enough, there are no other solutions to \eqref{DP}:

\begin{prop}\label{BBHtype}
There exists $\lambda_\beta^*>0$ so that for every $\lambda<\lambda_\beta^*$
the unique solution to \eqref{DP} is the symmetric solution $U_\lambda$.
\end{prop}

\begin{proof}
First, define the convex set $\mathcal B=\{\Psi\in\HH: \ |\Psi(x)|\le 1 \ \text{in ${\mathbb D_1}$}\}$.  By Lemma~4.2 of \cite{AB2}, any solution of \eqref{DP} lies in $\mathcal B$.
The second variation of the energy around $\Psi=(\psi_+,\psi_-)\in\HH$ in direction $\Phi=(\varphi_+,\varphi_-)\in H^1_0({\mathbb D_1};\CC^2)$ is:
\begin{multline*}
  E''_\lambda(\Psi)[\Phi] = 
    \int_{{\mathbb D_1}} \left\{ |\nabla\Phi|^2 + \lambda (|\Psi|^2-1)|\Phi|^2
       + 2\lambda \llan \Psi,\Phi\rran^2 \right. \\ 
       \left.
        + {\beta\lambda} \left[ |\psi_+|^2-|\psi_-|^2\right]
                     \left[ |\varphi_+|^2-|\varphi_-|^2\right]
         + {2\beta\lambda} \left[
                \llan\psi_+,\varphi_+\rran
                - \llan \psi_-,\varphi_-\rran\right]^2 \right\} .
\end{multline*}
For any $\Psi\in \mathcal B$, we have
$$  E''_\lambda(\Psi)[\Phi] \ge  \int_{{\mathbb D_1}} \left\{ |\nabla\Phi|^2 
- {C(\beta)\lambda}|\Phi|^2\right\} dx,  $$
with constant $C(\beta)\ge 0$ independent of $\lambda,\Phi$.  By choosing
$\lambda^*_\beta$ sufficiently small that ${C(\beta)\lambda}$ is smaller than the first Dirichlet eigenvalue of the Laplacian in ${\mathbb D_1}$ we may conclude that $E''_\lambda(\Psi)$ is a strictly positive definite quadratic form on $H^1_0({\mathbb D_1};\CC)$, for any $\Psi\in\mathcal B$.
Thus, $E_\lambda$ is strictly convex on $\mathcal B$, and hence it has a unique critical point.
\end{proof}

We also observe that the symmetric vortex solution $U_\la$ is the unique solution to \eqref{DP} for which both components vanish at the origin:

\begin{prop}\label{unique2}
Suppose $\Psi$ is a solution to \eqref{DP} with $\psi_+(0)=0=\psi_-(0)$. Then $\Psi=U_\la$.
\end{prop}

\begin{proof}
The proof is exactly as for Proposition~\ref{unique}:  as before, let $w_\pm={\psi_\pm\over \rho(|x|)}$, with $\rho(r)={1\over\sqrt{2}}|u_\la|$,
 so $w_\pm=e^{i\theta}$ on $\partial\mathbb{D}_1$.  In particular, $|w_\pm|^2=|\partial_\tau w_\pm|^2=1$ on $\partial\mathbb{D}_1$, so 
$$  F(1)= -\frac12\int_{\partial\mathbb{D}_1} |\partial_r w_+|^2 + |\partial_r w_-|^2\le 0. $$
Arguing as above, we may conclude that $G(1,0)=0$, and hence $|w_\pm|=1$ and 
$\partial_r w_\pm=0$ in $\mathbb{D}_1$, with $w_\pm=e^{i\theta}$ on $\partial\mathbb{D}_1$.  Thus, $w_\pm(x)= e^{i\theta}$ on $\mathbb{D}_1$, that is $\psi_\pm=\rho(r)e^{i\theta}={1\over\sqrt{2}}u_\la$.
\end{proof}

It is important to recognize the role played by two groups of symmetry acting on the problem.  First, the problem is invariant under the action of  the group $\mathbb{S}^1$, with the following representation:
for $\psi\in H^1_{e^{i\theta}}(\ddd;\mathbb{C})$, writing the independent variable in complex form $z=x+iy$, we represent $\xi\in \mathbb{S}^1\subset\CC$ via
$$   (R_\xi \psi) (z)= \overline{\xi} \psi(\xi\, z).
$$
By abuse of notation, for 
$\Psi=(\psi_+,\psi_-)\in \HH$ we
define
$$   R_\xi \Psi = \left( R_\xi\psi_+\, , \, R_\xi\psi_-\right).  $$
  Our problem is also invariant with respect to the involution,
\be\label{T}  TV = \left( -v_-(-x) \, , \,  -v_+(-x) \right), \qquad V=(v_+(x),v_-(x)).
\ee
The equivariant solution $U_\lambda$ is fixed by both groups, $R_\xi U_\lambda= U_\lambda = TU_\lambda$, for all $\xi\in \mathbb{S}^1$.

With the goal of studying bifurcations from the symmetric solutions, our first task is to study the spectrum of the linearization 
of the energy around $U_\lambda$,
$$   E''_\la (U_\la)[\Phi] =
     \int_{{\mathbb D_1}} \left[
      |\nabla\Phi|^2 +\la (|u_\la|^2-1) |\Phi|^2 
      + \la\llan u_\la, \varphi_+ + \varphi_-\rran^2
       + \la\beta \llan u_\la, \varphi_+ - \varphi_-\rran^2\right],
$$
for $\Phi=(\varphi_+,\varphi_-)\in H^1_0(\ddd; \CC^2)$.  The quadratic form $E''_\la(U_\la)$ is associated to the linearized operator,
$$  L_\lambda \Phi =
\left[ \begin{matrix}
   L_\lambda^+ \Phi \\ L_\lambda^- \Phi
\end{matrix}
\right]  =
\left[\begin{matrix}
 -\Delta \varphi_+ + \lambda(|u_\lambda|^2 -1) \varphi_+
  +\lambda \lan u_\lambda, \varphi_+ + \varphi_-\ran u_\lambda
    + \beta\lambda\lan u_\lambda, \varphi_+ - \varphi_-\ran u_\lambda \\
     -\Delta \varphi_- + \lambda(|u_\lambda|^2 -1)\varphi_-
  +\lambda \lan u_\lambda, \varphi_+ + \varphi_-\ran u_\lambda
    - \beta\lambda\lan u_\lambda, \varphi_+ - \varphi_-\ran u_\lambda
\end{matrix}\right].
$$
 For each $\lambda>0$, $L_\la$ defines a self-adjoint operator acting on its domain $H^2\cap H^1_0({\mathbb D_1};\mathbb C^2)\subset L^2({\mathbb D_1};\mathbb C^2)$ (with real scalar product).  Moreover, by elliptic regularity theory, for all $\la>0$ is has compact resolvent, and thus discrete spectrum, consisting of eigenvalues, 
$$ \sigma(L_\la) = \{ \mul \le \mu_2(\lambda) \le
\mu_3(\la) \le \cdots\},  $$
repeated according to their (finite) multiplicities, and ordered by the min-max principle for each fixed $\la>0$. .
We  seek the critical value $\la=\lab$ at which the radial solution loses stability with increasing $\la$.  
The group invariance under $R_\xi$, $\xi\in\mathbb{S}^1$, and $T$ are inherited by the linearized operator, and each of the eigenspaces is invariant under these two symmetries.

We will uncover the spectral properties of the linearized operator $L_\la$ via a sequence of reductions, and prove analytic dependence on $\la$ of eigenvalues and eigenfunctions provided they are simple modulo the action of the group $\mathbb{S}^1$.

%The following proposition summarizes the behavior of the ground state eigenvalue $\mul$ and the associated eigenspaces: 
%\begin{prop}\label{m2}
%\begin{enumerate}
%\item[(1)]  For every $\beta\ge 1$, $\mu_1(\lambda)>0$ for all $\la>0$.
%\item[(2)]
%For every $\beta\in (0,1)$, there exists $\lab>0$ so that 
%$\mul>0$ if $\la<\lab$, and $\mul<0$ when $\la>\lab$.  
%\item[(3)] For every $\beta\in (0,1)$, there exists $\la^*< \lab$ so that 
%for each $\la>\la^*$, the ground state eigenvalue $\mul$ of $L_\la$ is simple modulo the $\mathbb{S}^1$ action:  its eigenspace  is given by
%$$   \XL = \left\{ s\, (w_\xi, -w_\xi), \  w_\xi = R_\xi\, w_1 =
% \overline{\xi} \, a_0(r) - \xi\, a_2(r)\, e^{2i\theta}\ : \ s\in \RR, \ \xi\in \so\right\},$$
%for unique $a_0=a_0(r;\la)$, $a_2=a_2(r;\la)$, with $a_0(r;\la)\ge a_2(r;\la)\ge 0$ and $a_0(0;\la)=1$.  Moreover, $a_0(r;\la)>0$ in $[0, 1)$ and  $a_2(r;\la)\sim r^2$ near $r=0$, for all $\la>0$.
%\item[(4)]  There exists an open interval $\Lambda\ni\lab$ and $\eta>0$ so that the second eigenvalue $\mu_2(\lambda)$ of $L_\la$ satisfies $\mu_2(\lambda)>\eta$ for all $\lambda\in\Lambda$.
%\item[(5)]  The ground-state eigenvalue $\mu_1(\lambda)$ and the associated eigenvectors $W_\xi=(w_\xi,-w_\xi)\in\XL$ are real analytic functions of $\lambda\in\Lambda$.
%
%\end{enumerate}
%\end{prop}
%
\subsection*{Reduction of $L_\la$}

First, we reduce to a scalar problem:  let
$$  \LLL \varphi := -\Delta \varphi + \la\, (|u_\la|^2-1)\varphi
    + 2\beta \la\,\langle   u_\la,\varphi \rangle u_\la,  $$
as an operator acting on $L^2({\mathbb D_1}; \mathbb C)$, with associated quadratic form
$$  {\mathcal Q}_\la(w):= \int_{{\mathbb D_1}}\left[ |\nabla w|^2 + \la(|u_\la|^2-1)|w|^2
      + 2\beta\la \langle  u_\la,w  \rangle^2\right] . $$
We also define the linearization of the classical Ginzburg--Landau energy,
$$  G''_\la(u_\la)[\varphi] = \int_{{\mathbb D_1}} \left[
      |\nabla\varphi|^2 +\la (|u_\la|^2-1) |\varphi|^2 
      + 2\la\llan u_\la, \varphi\rran^2\right].
$$
and
$$  \LGL \varphi:= -\Delta\varphi
 + \la (|u_\la|^2-1)\varphi + 2\la \langle u_\la,\varphi   \rangle u_\la.  $$
 By  \cite{M1}, $\LGL$ is a  positive definite operator for all $\la>0$.

\begin{lem}\label{step1}
\begin{enumerate}
\item[(1)]
For any $\beta\ge 1$, $L_\la$ is  positive definite for all $\la>0$.  
\item[(2)]
For any $\beta>0$,  $\mu$ is an eigenvalue of $L_\la$ with eigenfunction $\Phi$ if and only if:
\begin{enumerate}
\item[(i)] either 
$\mu\in\sigma(\LGL)$; 
\item[(ii)] or $\mu\in\sigma(\LLL)$ and $\Phi=(\varphi,-\varphi)$, where $\varphi$ is an eigenfunction for $\LLL$.
\end{enumerate}
% $L_\la$ is strictly positive definite if and only if the quadratic form
%$$  Q_R(w):= \int_{D_R}\left[ |\nabla w|^2 + (|u_R|^2-1)|w|^2
%      + 2\beta (u_R,w)^2\right]  $$
%is positive definite over $w\in  H_0^1(D_R;\CC)$.  
%\item[(c)] $W$ is a ground-state eigenfunction for $E''_R(U_R)$ if and only if $W=(w, -w)$ with $w$ a ground-state eigenfunction for $Q_R$.
\end{enumerate}
\end{lem}

\begin{proof}
A simple calculation shows that
\be\label{id1}  E''_\la(U_\la)[\Phi] = G''_\la(u_\la)[\varphi_+] + G''_\la(u_\la)[\varphi_-]
      + (\beta-1)\la\int_{{\mathbb D_1}}   \llan u_\la, \varphi_+-\varphi_-\rran^2.  
\ee
By \cite{M1}, $G''_\la(u_\la)$ is positive definite, and thus \eqref{id1} implies that $L_\la$ is positive definite for $\beta\ge 1$, proving {\it (1)}.

Now let $\beta>0$ be arbitrary, and assume $\Phi=(\varphi_+,\varphi_-)$ solves $L_\la\Phi =\mu\Phi$.
We then have
\begin{align*}  
L_\la^+\Phi &= \LGL\varphi_+ + 
    (\beta-1) \la\llan u_\la,\varphi_+ - \varphi_-\rran u_\la \\
L_\la^-\Phi &= \LGL\varphi_- -
    (\beta-1) \la\llan u_\la,\varphi_+ - \varphi_-\rran u_\la .
\end{align*}
Adding these two identities together,
$$ \LGL(\varphi_+ +\varphi_-) = L_\la^+\Phi + L_\la^-\Phi = \mu(\varphi_+ +\varphi_-) .  $$
Thus, if $\mu\notin\sigma(\LGL)$, we must have $-\varphi_-=\varphi_+=:\varphi$, and moreover $L_\la^\pm\Phi = \LLL \varphi =\mu\varphi$.\\
This proves the only if part. The ``if'' part is obvious.
\end{proof}

Denote the eigenvalues (repeated by multiplicity) of $\LLL$ by
$$  \sigma(\LLL)=\{\tilde\mu_1(\la)\le\tilde\mu_2(\la)\le\cdots\}.  $$

%By the result of \cite{M1}, $\inf\sigma(\LGL)=:\tau(\lambda)>0$ for all $\la>0$, and is monotone decreasing in $\lambda$.  We thus have the following consequence of Lemma~\ref{step1}:
%
%\begin{cor}\label{corstep1}
%Suppose $\mu_1(\lambda_0)\le 0$ for some $\lambda_0>0$.  Then,
%for $\tau_0=\tau(\lambda_0)$ we have $\mu_1(\la)=\tilde\mu_1(\la)$ for every $\lambda\ge\lambda_0$.  The eigenfunctions of $L_\la$ have the form $\Phi=(\varphi,-\varphi)$, where $\varphi$ is an eigenfunction corresponding to $\LLL$.
%\end{cor}

%In particular, it will suffice to consider the operator $\LLL$ and its quadratic form $\mathcal{Q}_\la$ in the proof of Proposition~\ref{m2}.

%%%%%%%%%%

The next step is to decompose $w\in  H_0^1(D;\CC)$ in its Fourier modes in $\theta$:
$$
w=\sum_{n\in\ZZ} b_n(r) e^{in\theta}.
$$
  Using Parseval's identity, we have
$$  \int_{\ddd} \llan u_\la,w\rran^2 \, dx = {\pi\over 2} \int_0^1 
         f_\la^2(r) \sum_{n\in\ZZ} | b_{n+1} +\overline{b}_{1-n}|^2\,  r\, dr.  
$$
Consequently, the operator $\LLL$ can be identified to a direct sum in Fourier modes,
\begin{equation}\label{direct}
   \LLL w \cong \bigoplus_{n=0}^\infty \LLL^{(n)}(b_{n+1},b_{1-n}),
\end{equation}
where the operators $\LLL^{(n)}$ are associated to the quadratic forms
\begin{multline*}  \QQQ^{(n)}(b_{n+1},b_{1-n}):= \pi
\int_0^1 \biggl[ |b'_{n+1}|^2 + |b'_{1-n}|^2 
  + {(n+1)^2\over r^2}|b_{n+1}|^2 + {(1-n)^2\over r^2}|b_{1-n}|^2
   \\  
   + \la(f_\la^2-1)\left(|b_{n+1}|^2 + |b_{1-n}|^2\right)
   + \la\beta f_\la^2| b_{n+1} +\overline{b}_{1-n}|^2
\biggr] r\, dr,
\end{multline*}
for $n\ne 0$, and 
$$ \QQQ^{(0)}(b_1):= \pi\int_0^1 \left[ |b'_1|^2 + {1\over r^2} |b_1|^2
+ \la(f_\la^2-1)|b_1|^2 + \la\beta f_\la^2|b_1 +\overline{b_1}|^2
\right] r\, dr.
$$
If we let 
$$
\begin{aligned}
&X=\{ f:(0,1]\to\CC ; f(r)/r\in L^2, f'\in L^2,  f(1)=0\},\\
&Y=\{  f:(0,1]\to\CC ;  f'\in L^2,  f(1)=0\},
\end{aligned}
$$
then: $\QQQ^{(0)}$ acts on $X$, $\QQQ^{(1)}$ acts on $X\oplus Y$ and, for $n\neq 0, 1$, $\QQQ^{(1)}$ acts on $X\oplus X$.

The spectrum of this direct sum is given by the union of the eigenvalues of the operators $\LLL^{(n)}$.  As we show below, only one of these operators may contribute an eigenvalue near zero.  Define (in Fourier Space)
$$   \tilde\LLL \tilde w:= \bigoplus_{n\neq 1} \LLL^{(n)}(b_{n+1},b_{1-n}),
$$
where $\tilde w= \sum_{n\neq 0,2} b_n(r) e^{in\theta}$, and
so $\LLL w \cong \LLL^{(1)}(b_2,b_0) \oplus \tilde\LLL \tilde w.$  Let
$\tilde\QQQ$ denote the quadratic form associated to $\tilde\LLL$.

\begin{lem}\label{step2}  There exists a function $m_0(\lambda)>0$  such that $\inf \sigma(\tilde\LLL)\ge m_0(\lambda)>0$ for every $\lambda>0$ and, in addition, $\displaystyle \inf_I m_0(\la)>0$ for each compact interval $I\subset (0, \infty)$.
\end{lem}

\begin{proof}  This follows from the reductions described on page 337 of \cite{M1}.  Let $\tilde w\in  H_0^1(\ddd;\CC)$ with associated coefficients $\{b_n(r)\}_{n\neq 0,2}$, such that $\sum_{n\neq 0,2}\|b_n\|_{L^2((0, 1) ; r\, dr)}^2=1$.  We define $a_1:=i(\sum_{n\neq 0,2} |b_n|^2)^{1/2}$.  Then, $a_1(r)$ is purely imaginary,
$\|a_1\|_{L^2((0, 1) ; r\, dr)}=1$, and 
\begin{align*} \tilde\QQQ(\tilde w)&\ge
 \pi\sum_{n\neq 0,2} \int_0^1 \left\{ |b'_n(r)|^2
          + {n^2\over r^2} |b_n(r)|^2 +\lambda (f_\la^2(r)-1)|b_n(r)|^2
                                \right\} r\, dr \\
        & \ge \pi\int_0^1 \left\{ |a'_1(r)|^2 + {1\over r^2} a_1^2(r) 
             + \la(f_\la^2(r)-1)|a_1(r)|^2 \right\} r\, dr 
              \\ &= \frac12 Q_\la^{(0)}(a_1,a_1).
\end{align*}
Now set $m_0(\lambda):=\inf \frac12 Q_\lambda^{(0)}(a_1,a_1)$, where the infimum is taken over all $a_1\in H^1_0((0,1))$ with $\|a_1\|_{L^2((0,1);r\,dr)}=1$.  In  \cite{M1} it is proven that $m_0(\la)>0$ for all $\la>0$. 
By the min-max principle, $\inf\sigma(\tilde\LLL)\ge m_0(\la)>0$. 

It remains to show that $m_0(\lambda)$ can be bounded away from zero when $\la$ is bounded away from zero.  
 To do that, we make a transformation as in \cite{M1} in order to consider a fixed system of equations on an increasing family of disks ${\mathbb D}_R$, with $R=\sqrt{\lambda}$.  We define 
$\hat u_R(x):= u_\lambda(x/\sqrt{\lambda})=F(r,R)e^{i\theta}$, and $\hat a_1(r)=a_1(r/\sqrt{\lambda}), $
for $r\in [0,R]$.  In this way,
$$  {m_0(\lambda)\over\la}=\hat m_0(R):=
  \inf_{\|\hat a_1\|_{L^2((0,R) ; r\, dr)}=1}
    \pi \int_0^R \left\{ |\hat a'_1|^2 + {1\over r^2} \hat a_1^2(r) 
             + (F(r,R)^2(r)-1)|\hat a_1|^2 \right\} r\, dr .
$$
We now show $\hat m_0(R)$ is  decreasing in $R>0$.
First, the radial profiles $f_R(r)$ are pointwise decreasing in $R$:
whenever $R<R'$, we have $F(r,R')<F(r,R)$ for all $r\in (0,R)$. (Indeed, on $[0,R]$, $ F(\cdot,R')$ is a subsolution for  the equation satisfied by $F(\cdot,R)$.)  Using this fact and the inclusion $H^1_0({\mathbb D}_R)\subset H^1_0({\mathbb D}_{R'})$, we find that 
$\hat m_0(R')\le \hat m_0(R)$. In case the two are equal, this would imply that the minimizers of $Q_\lambda^{(0)}$ in ${\mathbb D}_{R'}$ (which are nonnegative,) vanish identically in ${\mathbb D}_{R'}\setminus{\mathbb D}_{R}$. This is impossible, by the maximum principle. Thus $\hat m_0(R)$ is strictly decreasing.

Finally, given a fixed interval $I=[\Lambda_0,\Lambda_1]$, whenever $\la\in I$,
$$  m_0(\la)= \la\hat m_0(\sqrt{\la}) \ge \Lambda_0 \hat m_0(\sqrt{\Lambda_1}):= C_I>0.  $$
\end{proof}

Combining the results of Lemmas~\ref{step1} and \ref{step2}, we may conclude:

\begin{cor}\label{cor1}
If $\mu_k(\la)\in\sigma(L_\la)$ and $\mu_k(\la)\le 0$ for some $\la>0$, then $\mu_k(\la)\in\sigma(\LL_\la^{(1)})$.
\end{cor}

In particular, if $\mu_1(\lambda)$ is to cross zero it must be because of the ground state eigenvalue $\mu_1^{(1)}(\lambda)$ of $\LLL^{(1)}$.  We perform one final reduction of the operator $\LLL^{(1)}$:
  We define a quadratic form $Q^{(1)}_\la$ on real-valued radial functions $(a_0,a_2)$ by
$$   Q^{(1)}_\la(a_0,a_2):= 
\pi\int_0^1 \left[ (a'_0)^2 + (a'_2)^2 + {4\over r^2}a_2^2
 + \la (f_\la^2-1)\left(a_0^2+a_2^2\right) 
   + \beta\lambda f_\la^2 \left(a_0 - a_2\right)^2\right] r\, dr.
$$
The  self-adjoint operator associated to $Q^{(1)}_\la$ is
\begin{equation}\label{Msys}
\mathcal{M}_\la\left[ \begin{matrix} a_0 \\ a_2 \end{matrix}\right]
=
\left[ \begin{gathered}
-a''_0 - {1\over r}a'_0 + \la (f_\la^2-1) a_0 + \beta\la f_\la^2 (a_0-a_2)
\\
-a''_2 - {1\over r}a'_2 + {4\over r^2}a_2 + \la (f_\la^2-1) a_2 - \beta\la f_\la^2 (a_0-a_2)
\end{gathered}\right].
\end{equation}

\begin{lem}\label{step3}
$\mu\in\RR$ is an eigenvalue of $\LLL^{(1)}$ over $L^2(([0,1]; r\, dr) ; \CC^2)$ if and only if it is an eigenvalue of $\mathcal{M}_\la$ over $L^2(([0,1]; r\, dr) ; \RR^2)$.  Moreover, if $\mu$ is a \underbar{simple} eigenvalue of $\mathcal{M}_\la$ with eigenspace spanned by $(a_0,a_2)$, then 
$$  \ker(\LLL^{(1)}-\mu I) =\{ t\, (\bar\xi a_0 \, , \, -\xi a_2): \ \xi\in \so, \ t\in\RR\}.  $$
\end{lem}

\begin{proof} 
Let $\mu\in\sigma(\LLL^{(1)})$ with (complex-valued) eigenvector $(b_0,b_2)$, that is:
\begin{equation}\label{Lsys}
\LLL^{(1)}\left[ \begin{matrix} b_0 \\ b_2 \end{matrix}\right]
=
\left[ \begin{gathered}
-b''_0 - {1\over r}b'_0 + \la (f_\la^2-1) b_0 + \beta\la f_\la^2 (b_0-\overline{b_2})
\\
-b''_2 - {1\over r}b'_2 + {4\over r^2}b_2 + \la (f_\la^2-1) b_2 + \beta\la f_\la^2 (b_2 - \overline{b_0})
\end{gathered}\right]
=\mu \left[ \begin{matrix} b_0 \\ b_2 \end{matrix}\right].
\end{equation}
We observe that $a_0=\Im b_0$, $a_2=\Im b_2$ will be eigenvectors of $\mathcal{M}_\la$ with the same eigenvalue $\mu$.
On the other hand, it is clear that if $(a_0,a_2)$ are (real-valued) eigenvectors of $\mathcal{M}_\la$, then $(b_0,b_2)=(ia_0,ia_2)$ will be eigenvectors of $\LLL^{(1)}$ with the same eigenvalue.  Thus,
$\sigma(\LLL^{(1)})=\sigma(\mathcal{M}_\la)$.

Finally, suppose $\mu$ is a simple eigenvalue of $\mathcal{M}_\la$ with eigenspace spanned by $(a_0,a_2)$.  If $(b_0,b_2)$ is an eigenfunction of $\LLL^{(1)}$, then (by the observation above) $(\Im b_0, \Im b_2)= -q(a_0,a_2)$ for $q\in\RR$.  Similarly, $(\Re b_0, -\Re b_2)$ is an eigenfunction of  $\mathcal{M}_\la$, and so $(\Re b_0, -\Re b_2)= p(a_0, a_2)$ for $p\in\RR$.  Setting
$t=\sqrt{p^2+q^2}$ and $\displaystyle\xi={p+iq\over t}\in \so$, 
we have 
$(b_0, b_2) = t\left(\bar\xi\, a_0 \, , \, -\xi\, a_2\right)$, as claimed.
\end{proof}

We conclude this part with the following essential fact about the ground state eigenvalue of $\mathcal{M}_\la$:

\begin{lem}\label{step4}
The ground state eigenvalue $\mu^{(1)}_1(\la)$ of $\mathcal{M}_\la$ 
is simple.  It is generated by $(a_0(r;\la), a_2(r;\la))$ with $0\le a_2(r;\la)\le a_0(r;\la)$, $a_0(r;\la), a_2(r;\la)>0$ in $(0,1)$, $a_0(0;
\la)>0$, and $a_2(r;\la)=O(r^2)$ for $r\to 0$, for all $\la>0$.
\end{lem}

\begin{proof}
This fact follows as in \cite{M1}:  we claim that, up to a change of sign,  $0\le a_2(r)\le a_0(r)$ holds for all $r$.  Indeed, if not we define
$\tilde a_2(r)=\min\{|a_0(r)|,|a_2(r)|\}$, and similarly $\tilde a_0(r)=\max\{|a_0(r)|,|a_2(r)|\}$.  Replacing $a_0$, $a_2$ respectively by $\tilde a_0$, $\tilde a_2$ does not change  the quantity 
$\displaystyle \int_0^1 [|a_0(r)|^2 + |a_2(r)|^2]\, r\, dr$, 
and the first, second, and fourth terms in $Q^{(1)}_{\la}$ are unchanged. However, the third and last terms are reduced, contradicting the minimality of the Rayleigh quotient at $(a_0,a_2)$.  By a standard argument, nonnegativity of the eigenfunctions and of $a_0-a_2$ implies simplicity of the eigenvalue.  Each function is strictly positive in $(0,1)$ by the strong maximum principle (or the uniqueness theorem for ordinary differential equations).

The behavior of $a_0,a_2$ at $r=0$ follows from the ordinary differential equations (see (16) of \cite{M1}) satisfied by $a_0, a_2$ in $(0,1)$.
\end{proof}

\subsection*{Analyticity}

In this part, we prove that the radial profile $f_\la$ is analytic in both $\la$ and $r$, and conclude real analytic dependence on $\la$ of the simple eigenvalues of $L_\la$.

\begin{prop}
\label{fanalytic}
$f_\lambda(r)$ is real-analytic in $r\in [0,1]$ and $\lambda>0$.
\end{prop}
As in the proof of Lemma~\ref{step2}, we rescale our problem to study a fixed equation in a variable domain.  We define 
$$F(r,R)=f_\la(r/\sqrt{\la}),$$
 for  $r\in (0,R)$, with $R=\sqrt{\lambda}$.  Then, $F$ solves
\begin{equation}\label{bvp}
   -F''-\frac1r F' + {1\over r^2} F = (1-F^2)F, \quad F(0)=0, \ F(R)=1.  
\end{equation}
While the existence of such an $F(r)$ may easily be done by minimization of its energy functional, to obtain the desired properties of $F$ it will be necessary to relate the solution $F(r,R)$ of the boundary value problem to solutions $\phi(r,b)$ of a Cauchy-type problem,
\begin{equation}\label{cauchy}
-\phi_{rr} - \frac1r \phi_r + {1\over r^2} \phi = (1-\phi^2)\phi,  \qquad
\phi(0,b)=0, \ \phi'(0,b)=b,
\end{equation}
where the value of $b$ is chosen (by ``shooting'') to achieve the boundary condition $\phi(R,b)=1$.  Indeed, it has been shown (see \cite{CEQ}) that $F(r,R)=\phi(r,b_R)$ for a unique value of $b=b_R>0$.
We note that the equation being singular at $r=0$, this is not a regular initial-value problem, and thus the existence and analyticity of the solution do not follow directly from the Picard existence theorem.  (See Theorems~ 8.1 and 8.3 of \cite{CL}.)  

In the remainder of this section it will be convenient to extend $\phi(r,b)$ to $r\in\CC$, with complex parameter $b\in\CC$.  
The following equivalence follows easily from the variation of parameters formula, and may be established by direct calculation.

\medskip

\begin{lem}\label{varpar} Let $g(r)$ be continuous for $|r|\le r_0$, $r_0>0$.  If $f$ is continuous on $|r|\le r_0$ and solves
\be\label{IE}   f(r) = b r + \frac12 \int_0^r \left({s\over r}-{r\over s}\right) s\,g(s)\, ds,  
\ee
then $f\in C^2(\{r\in\CC: \ 0<|r|<r_0\})$ and solves
\be\label{DE}  -f_{rr} - \frac1r f_r + {1\over r^2} f = g(r), \quad \text{for} \ r\in (0,r_0), \qquad
f(0,b)=0, \ f'(0,b)=b.
\ee
Conversely, if $f\in C^1(\mathbb{D}_{r_0})\cup C^2(\mathbb{D}_{r_0}\setminus\{0\})$ is a solution of \eqref{DE}, then it also solves \eqref{IE}.  
Moreover, if $g$ is analytic in $\mathbb{D}_{r_0}$, then $f$ will be analytic for $r\in \mathbb{D}_{r_0}$ and $b\in\CC$.
\end{lem}

The integral in \eqref{IE} is a complex path integral, and the path is the straight-line segment joining $0$ to $r$ in $\CC$.  Note that in the analytic case, the singularity of the solution of \eqref{DE} is removable, and thus $f$ is analytic in the entire disk $\mathbb{D}_{r_0}$.

Using Lemma~\ref{varpar} we may obtain an analytic solution to the Cauchy-like problem \eqref{cauchy} by Picard iteration of the integral equation,
\begin{equation}\label{inteq}  \phi(r,b) = 
  b r + \frac12 \int_0^r \left({s\over r}-{r\over s}\right) s\,
   (1-\phi(s,b)^2)\phi(s,b)\, ds.  
\end{equation}
Let $b_0\in\CC$ be fixed, and define a rectangle in $\CC\times\CC$,
$$  \mathcal{R}:=\left\{ (r,b) : \  |r|\le \rho, \ |b-b_0|\le\rho\right\}\quad
\text{with $\rho:=\min\left\{\frac12, {1\over 2|b_0|}\right\}$.} $$

\begin{lem}\label{picard}  For any fixed $b_0\in\CC$, there exists a unique solution to \eqref{cauchy}, which is analytic for $(r,b)\in\mathcal{R}$.  
\end{lem}

\begin{proof}
Let $h(r,b)=\phi(r,b)/r$.  Then, $f$ solves \eqref{inteq} if and only if $h$ solves the fixed-point equation,
$$  h(r,b) = T[h](r,b):=
  b  + \frac12 \int_0^r \left({s^2\over r^2}-1\right)
   s(1-s^2h(s,b)^2)h(s,b)\, ds.  
$$
Define a class of continuous functions,
$$  M:=\left\{ h\in C(\mathbb{D}_\rho;\CC): \ |h(r)-b_0|\le 1 \ \forall r\in \mathbb{D}_\rho\right\}.  $$
We observe that if $h\in M$, then for any $r\in \mathbb{B}_\rho,$
\be\label{rh}
|rh(r)| \le |r(h-b_0)| + |rb_0| \le \frac12 + \frac12 =1.
\ee
Thus, $|1-s^2h^2(s)|\le 2$ for $0<|s|<|r|\le\rho$, and 
$$   \left| T[h](r) - b_0\right| \le |b-b_0|+
  \frac12 \int_0^{|r|} |1-s^2h^2(s)| \, |sh(s)|\, d|s| \le \frac12+|r|\le 1, 
$$
and hence $T: \ M\to M$.  If $h_1,h_2\in M$, then
\begin{align*}
|T[h_2](r)-T[h_1](r)| &\le
   \frac12 \int_0^{|r|} \left[ 1 + |s|^2( |h_1|^2 + |h_1h_2| + |h_2|^2)\right] |s|\, |h_2(s)-h_1(s)|\, d|s| \\
   &\le  2 \|h_2-h_1\|_\infty \int_0^{|r|} |s|\, d|s| \\
   &= |r|^2 \,  \|h_2-h_1\|_\infty \le \frac14 \|h_2-h_1\|_\infty,
\end{align*}
for any $r\in \mathbb{D}_\rho$, where we have used \eqref{rh} to estimate the integrand.  Thus, $T$ is a contraction on $M$, and there exists a unique continuous solution of the fixed point equation for $r\in \mathbb{D}_\rho$.  Since the solution may be characterized as the uniform limit of the iterates $h_n:= T[h_{n-1}]$, $h_0=b$, and by induction each $h_n$ is analytic for $(r,b)\in \mathcal{R}$, the solution $h$ so obtained is analytic for $(r,b)\in\mathcal{R}$.  Setting $\phi(r,b)=rh(r,b)$, we obtain a unique analytic solution to \eqref{inteq}, which by Lemma~\ref{varpar} provides a unique analytic solution of \eqref{cauchy} for $(r,b)\in \mathcal{R}$.  
\end{proof}

For $r\neq 0$, the initial value problem for the differential equation \eqref{cauchy} is regular, and the existence, uniqueness, and analytic dependence of solutions in the complex plane follow from Theorem~8.3 of \cite{CL}.  Using this observation, we may extend the solution $\phi(r,b)$ analytically along the real axis in $r$ to obtain a maximally defined analytic solution.  Indeed, for any $b$, $|b-b_0|<\rho$, define $R_b$ to be the supremum of all real values $R>0$ for which there exists an analytic solution $\phi(r,b)$ of \eqref{cauchy} in a $\CC$-neighborhood of the real interval $[0,R]$.  By Lemma~\ref{picard}, we know that $R_b>\rho$.

\begin{prop}\label{extension}  For any $b\in\RR$ there exists a unique solution $\phi(r,b)$ to \eqref{cauchy} for $r\in [0,R_b)$, which is real analytic in $(r,b)$.  
Either $R_b=\infty$, or $R_b<\infty$ and $\lim_{r\to R_b^-} |\phi(r,b)|=\infty$.
\end{prop}

\begin{proof}
The proof is a standard extension argument from the theory of ODE (see \cite{CL}.)  By the definition of $R_b$, for each real value $s\in [0,R_b)$ there exists a neighborhood $N_s$ of $[0,s)$ for which $\phi(r,b)$ is analytic for $r\in N_s$.  By the uniqueness of solutions to \eqref{cauchy} these sets are nested, and in fact $\phi(r,b)$ is analytic in $\mathcal{N}=\cup_{s\in [0,R_b)} N_s$.
Assume $R_b<\infty$, but 
$$  \limsup_{r\to R_b^-\atop r\in\mathcal{N}} |\phi(r,b)|<\infty.
$$
Then, from the integral equation \eqref{inteq}, it is easy to see that $\phi, \phi'$ both have limits as $r\to R_b^-$.   Using these limits as initial conditions for the differential equation, and applying Theorem~8.1 of \cite{CL}, there exists a $\CC$-neighborhood of $R_b$ and an analytic solution of the equation, which extends $\phi(r,b)$.  This contradicts the definition of $R_b$ as the supremum.  By restricting $b$ and $r$ to $\RR$, the complex analytic solution is real analytic on the desired domains.
\end{proof}

We next prove some additional properties of $F(r,R)=\phi(r,b_R)$.  In the following we restrict to $r\in\RR$.

\begin{lem}\label{L123}
Suppose $F(r)$ is a solution of \eqref{cauchy} with $0<F(r)<1$ on $(0,r_0)$.  Then:
\begin{equation}
F''(0)=0; \qquad  \left[{F(r)\over r}\right]'<0; \qquad
  {2F(r)\over r^3}>{F'(r)\over r^2}, \quad r\in (0,r_0).
\end{equation}
\end{lem}

\begin{proof}
Since $F(r)/r\to b$ as $r\to 0$, by the equation and L'H\^opital's rule,
$$   \lim_{r\to 0^+} F''(r) = 
 \lim_{r\to 0^+} {r F'(r) - F(r)\over r^2} = \frac12 \lim_{r\to 0^+}F''(r). $$
Hence, $F''(0)=0$.  To verify the second conclusion, let $h(r)=F(r)/r$, and calculate 
$$  {1\over r^3}\left(r^3 h'(r)\right)'=h''(r) + {3\over r} h'(r) = {1\over r} F(r)(F^2(r)-1) <0.  $$
In particular, $r^3 h'(r)$ is strictly decreasing.  Since $r^3 h'(r)|_{r=0}=0$, we have $h'(r)=(F(r)/r)'<0$, as claimed.

Finally, for the third statement, 
$$  {F'\over r^2} - {2F\over r^3} = \left({F\over r^2}\right)'
  = \left({h\over r}\right)' = {h'\over r} - {h\over r^2} <0, $$
by the second conclusion.
\end{proof}

As the boundary-value problem \eqref{bvp} does admit a unique solution for each $R>0$ (obtained as an energy minimizer,)  for each $R>0$, there is a unique choice of $b=b_R>0$ for which $\phi(R,b_R)=1$.  The following facts are well-known, but we sketch the proof here:

\begin{lem}\label{P-lem5}
Let $\phi(r,b_R)$ solve \eqref{bvp}.  Then $b_R>0$,  $0<\phi(r,b_R)<1$ and $\partial_r \phi(r,b_R)\ge 0$ for $r\in (0,R)$.
\end{lem}  

\begin{proof} 
The existence of a solution of \eqref{bvp} follows from minimizing the energy,
$$   E(F)= \int_0^R \left[ \frac12 (F'(r))^2 + {1\over r^2} F^2(r)
       + \frac14 (F^2(r)-1)^2\right] r\, dr,  $$
over $F\in H^1((0,R) ; r\, dr)$ with $F(0)=0$, $F(R)=1$.   Since  $E(|F|)=E(F)$, $|F|$ is a minimizer if $F$ is. By the strong maximum principle, we find that   $|F|(r)>0$, and thus $F>0$ in $(0,R]$.  Applying the maximum principle to the equation satisfied by $F^2(r)-1$ we may also conclude that $F(r)<1$ for $r\in (0,R)$. We next prove that this $F$ is  monotone.  First, we claim $b=F'(0)>0$.  Indeed, if $F'(0)=0$, then $F$ is a fixed point of the equation \eqref{inteq} with $b=0$.  Since the solution is unique for each fixed $b$, we must have $F(r)\equiv 0$, which is a contradiction.  Thus, $b>0$. Suppose that $F$ is not monotone.  Then, there exist $0<R_1<R_2<R$ such that $F$ has a local maximum at $R_1$ and a local minimum at $R_2$, with $F(R_1)>F(R_2)$.  Therefore, we have $F'(R_1)=0=F'(R_2)$, and 
$F''(R_1)\le 0\le F''(R_2)$, and so
$$  {F(R_1)\over R_1^2} \le F(R_1)\left(1-F^2(R_1)\right), \qquad
\text{and}\qquad
  {F(R_2)\over R_2^2} \ge F(R_2)\left(1-F^2(R_2)\right).
$$
Consequently, we have
$$  R_1^{-2} \le 1-F^2(R_1) \le 1-F^2(R_2) \le R_2^{-2}, $$
a contradiction.  Hence, $F(r)$ is monotone, and $F'(r)=\partial_r \phi(r,b_R)\ge 0$ for all $r\in (0,R)$.

Finally, uniqueness of the solution of \eqref{bvp} is proved in \cite{M2}.

\end{proof}

We now consider the linearized operator,
$$  Lg(r):= -g'' -\frac1r g' + {1\over r^2} g - (1-3F^2(r,R)) g,
$$
around $F(r,R)=\phi(r,b_R)$, the solution of \eqref{bvp}.

\begin{lem}\label{P-lem4}
Let $\lambda_1$ denote the smallest eigenvalue of $L$ with Dirichlet boundary conditions on $(0,R)$. Then $\lambda_1>0$.
\end{lem}

\begin{proof}
Let $g\in H^1_0((0,R))$, $g(r)>0$, be an eigenfunction associated to the first eigenvalue $\lambda_1$, $Lg=\lambda_1 g$.
Now, $h:= F'(r)>0$ solves a similar equation, and in fact
$$ Lh = {2F\over r^3}-{F'\over r^2} >0,
$$
by Lemma~\ref{L123}.  Multiplying the equation for $g$ by $rh$, the equation for $h$ by $rg$, and integrating by parts, we have
$$ \lambda_1 \int_0^R g(r)\, h(r)\,r\, dr = 
   \int_0^R \left( {2F\over r^3}-{F'\over r^2}\right)g(r)\, r\, dr  >0.
$$
In particular, $\lambda_1>0$.
\end{proof}

We may now examine the dependence of the solution $\phi(r, b)$ on the shooting parameter $b$:

\begin{lem}\label{P-lem6}
Let $b=b_R$, with $\phi(R,b_R)=1$.  Then $\partial_b \phi(r, b_R)>0$.
\end{lem}

\begin{proof}
Let $h(r)=\partial_b \phi(r,b)$.  Then we have 
\be\label{heq}
  -h'' - \frac1r h' + {1\over r^2} h = (1-3\phi^2(r,b_R))h, \qquad h(0)=0, \quad h'(0)=1. 
\ee
Suppose that $h$ vanishes somewhere in $(0,R]$.  Let $R_0\in (0,R)$ be the smallest value of $r>0$ for which $h(R_0)=0$.  Since $h'(0)=1$, we must have $R_0>0$ and $h(r)>0$ in $(0,R_0)$.  Choosing $g(r)>0$ the eigenfunction associated to the smallest eigenvalue $\lambda_1$ of $L$, 
$$  -g'' -\frac1r g' + {1\over r^2} g = (1-3\phi^2(r,b_R)) g + \lambda_1 g,
$$
we multiply by $rh$ and integrate by parts on $(0,R_0)$ to obtain:
$$  \lambda_1\int_0^{R_0} g\, h\, r\, dr = 0.  $$
This is impossible, as $g,h>0$ in $(0,R_0)$, and hence $h(r)=\partial_b \phi(r,b_R)>0$ in $(0,R]$.
\end{proof}

\begin{lem}\label{lem7}
$b_R$ is analytic in $R>0$, and $\partial_R b_R<0$.
\end{lem}

\begin{proof}
Since  $b_R$ is defined as the solution to the equation $\phi(R,b)=1$, and from Lemma~\ref{P-lem6} we have $\partial_b \phi(R, b_R)>0$, applying the analytic version of the Implicit Function Theorem (see \cite{Nir}), we conclude that both 
$\partial_R b_R=-\displaystyle{\partial_R \phi(R,b_R)\over \partial_b \phi(R,b_R)}<0$ and the dependence of $b_R$ on $R$ is real-analytic.
\end{proof}

Proposition~\ref{fanalytic} now follows trivially from Lemma~\ref{extension} and Lemma~\ref{lem7}, as $f_\la(r)=F(\sqrt{\la}r,\sqrt{\la})$, and $F(r,R)=\phi(r,b_R)$ is the composition of analytic maps in a neighborhood of the positive real axis $\la>0$.

\medskip

We may now apply analytic perturbation theory to simple eigenvalues of the operator $\mathcal M_\la$ associated to the quadratic form $Q_\la$.

\begin{lem}\label{step6} Assume that, for some $\la_0>0$ and $n\ge 1$, $\mu_n^{(1)}(\la_0)$ is a simple eigenvalue of $\mathcal{M}_{\la_0}$.  Then:
\begin{enumerate}
\item[(a)]  There exist $\delta,\eta>0$ such that for $\la$ in a (complex) neighborhood $\mathbb{D}_\delta(\la_0)$, the operator $\mathcal{M}_\la$ has exactly one isolated simple eigenvalue $\mu_n^{(1)}(\la)\in \mathbb{D}_\eta(\mu_n^{(1)}(\la_0))$.
\item[(b)]  There exists a normalized eigenvector $(a_0(\cdot;\la),a_2(\cdot;\la))$ of $\mathcal{M}_\la$ with eigenvalue $\mu_n^{(1)}(\la)$, each depending analytically on $\lambda\in \mathbb{D}_\delta(\la_0)$.
\end{enumerate}
\end{lem}

We remark that, by Lemma~\ref{step4}, the ground state eigenvalue $\mu_1^{(1)}(\la)$ is a simple eigenvalue of $\mathcal{M}_\la$ for each fixed $\la>0$, and so the conclusions of Lemma~\ref{step6} apply to $\mu_1^{(1)}(\la)$ in particular.

\begin{proof}
First, by Proposition~\ref{fanalytic}, $f_\la(r)$ is real-analytic in both $r$ and $\la>0$, and thus may be extended to complex $\la$ as an analytic function of both $r,\la$, for $\la\in \mathbb{D}_{\delta_0}(\la_0)$ for some $\delta_0>0$.

Next, we observe that, for $\la\in \mathbb{D}_{\delta_0}(\la_0)$, $\mathcal{M}_\la$ is an analytic family in the sense of Kato, which has compact resolvent for all $\la\in \mathbb{D}_{\delta_0}(\la_0)$ and is self-adjoint for $\la>0$.    The conclusions {\it (a)} and {\it (b)} of the lemma then follow from the Kato--Rellich Theorem (Theorem XII.8 of \cite{RS}).  
\end{proof}

\subsection*{Monotonicity of simple eigenvalues}

In order to study the dependence on $\la$ of the eigenvalues of the linearized operator $\mathcal{M}_\la(a_0,a_2)$ (defined in \eqref{Msys}), we proceed as in the proof of Lemma \ref{step2}: we replace the dependence on $\la$ by a dependence on the domain $(0,R)$, via a change of variables.  In this way, we define the quadratic form
$$  \hat Q_R(\hat a_0,\hat a_2)= Q^{(1)}_\la(a_0,a_2), \qquad
\text{where} \ \hat a_0(r)= a_0(rR), \ \hat a_2(r)=a_2(rR), \ R=\sqrt{\la}.
$$
The associated operator is then denoted by $\widehat{\mathcal M}_R$.
We observe that the eigenvalues 
$\sigma(\widehat{\mathcal M}_R)=\{\hat\mu_n(R)\}_{n\in\NN}$ (ordered by the min-max principle, and repeated by multiplicity,)
are related to the eigenvalues $\mu_n^{(1)}(\la)$ via 
$\mu_n^{(1)}(\la)=R^2 \hat\mu_n(R)$.

\begin{prop}\label{minmax}
Suppose $\hat\mu_n(R_0)$ is a simple eigenvalue of $\widehat{\mathcal M}_{R_0}$ for some $R_0>0$.  Then $\hat\mu'_n(R_0)<0$.
\end{prop}

\begin{proof}
We use the family $\phi(r,b)$ of solutions to the Cauchy problem \eqref{cauchy} above, and recall that $\phi(r,b)$ is real analytic in both $(r,b)$.

Let $R_0>0$ be fixed, and $\eta>0$ given.  We recall that for any $R>0$ there exists a unique $b=b_R>0$ such that the solution to the boundary-value problem \eqref{bvp} is $F(r,R)=\phi(r,b_R)$.  Furthermore, we assert that:
\begin{gather}
\label{a} \text{there exists $\eps>0$ so that $|b_R - b_{R_0}| <\eta$ whenever $|R-R_0|<\eps$;} \\
\label{b} \text{there exists $C_0>0$ so that $b'(R)\le -C_0<0$ whenever
$|R-R_0|<\eps$.}
\end{gather}
Indeed, both follow from the conclusions of Lemma~\ref{lem7} and Proposition~\ref{fanalytic}.

Next, we claim that there exists a constant $C_1>0$ so that
\be\label{N1}
\phi(r,b_R)\ge C_1 \, r,\qquad\text{for all $r\in [0,R],$ $R\in (R_0-\eps,R_0+\eps)$.}
\ee
From statement \eqref{a} above, and the analyticity of $\phi$, there exists $\delta>0$ such that
$$
\phi'(r,b_R) \ge {b_{R_0}\over 2} \ \text{ for all $r\in (0,\delta)$ and $|R-R_0|<\eps.$}
$$
Since $\phi(r, b_R)>0$ for $r>0$ and for all $R>0$, 
$$  k:=\min_{r\in[\delta,R]\atop |R-R_0|\le\eps} \phi(r, b_R) >0. $$
Let $C_1:= \min\left\{ {b_{R_0}\over 2}, {k\over R_0}\right\}$.  Then, 
putting the previous two estimates together we conclude that $F(r,R)=\phi(r,b_R)\ge C_1\, r$, for all $r\in [0,R]$ and $|R-R_0|<\eps$, and the claim \eqref{N1} is established.

The next step involves the derivative $\partial_b \phi(r, b_r)=: h(r,R)$.  By Lemma~\ref{P-lem6}, $h$ is analytic in both $(r,R)$, and $h(r,R)>0$ for all $r\in (0,R]$, $b>0$.  We recall that $h$ solves \eqref{heq}, and $h'(0,R)=1$ for all $R$.  

Following exactly the same arguments used in proving \eqref{N1}, we  obtain that the existence of a constant $C_2>0$ so that
\be\label{N2}
h(r,R)=\partial_b \phi(r,b_R) \ge C_2\, r, \qquad \text{for all $r\in [0,R],$ $R\in (R_0-\eps,R_0+\eps)$.}
\ee
As a consequence of \eqref{N2} and \eqref{b}, we have
\begin{align*}
\partial_R F(r,R) = \partial_R \phi(r, b_R) = 
\partial_b \phi(r, b_R)\, \partial_R b_R 
\le -C_0 C_2 r,
\end{align*}
for all $r\in [0,R]$ and $|R-R_0|<\eps$, using statement \eqref{b} above.
By the mean-value theorem, for any $R\in (R_0,R_0+\eps)$ and for all $r\in [0,R_0]$, there exists $\tilde R\in (R_0,R)$ with
\begin{align}\nonumber
F^2(r,R_0) - F^2(r,R) &= -2 (R-R_0) F(r,\tilde R)\partial_R F(r,\tilde R) \\
&\ge (R-R_0) C_3\, r^2,
\label{N3}
\end{align}
with constant $C_3=2C_0C_1 C_2>0$.

We are now ready to bound the eigenvalue from below.  Let $R\in (R_0, R_0+\eps)$ be fixed.  By Lemma~\ref{step6}, there exists $\delta>0$ for which $\hat\mu_n(R)$ is simple for $|R-R_0|<\delta$, and $\hat\mu_n(R)$ is analytic in that interval.
Let $E_n^0$ denote the linear span of the first $n$ eigenfunctions,
$w_1=(a_0^1,a_2^1),\dots,w_n=(a_0^n, a_2^n)$, of  $\widehat {\mathcal{M}}_{R_0}$, each normalized with $\|w_k\|_{L^2}=1$, $k=1,\dots,n$.  As each $w_k(R_0)=0$, extending their definition by zero for $r>R_0$, each lies in the domain of the the operator $\widehat {\mathcal{M}}_{R}$ for all $R>R_0$.
Thus, by the Courant-Fischer min-max principle, 
\begin{align*}
\hat\mu_n(R) &= \inf_{\dim E=n} \max_{w\in E\atop \|w\|_2=1}
     \hat Q_R (w) \\
&\le \max_{w\in E_n^0\atop \|w\|_2=1}
     \hat Q_R (w)  \\
 &=\max_{w=(a_0,a_2)\in E_n^0\atop \|w\|_2=1}
 \left[
 \hat Q_{R_0}(w) +
   \int_0^{R_0} \left[ F^2(r,R)-F^2(r,R_0)\right] 
     \left[ (a_0^2 +a_2^2) + \beta(a_0-a_2)^2\right] r\, dr 
 \right] \\
 &\le \hat\mu_n(R_0) -  C_3 (R-R_0) \min_{w=(a_0,a_2)\in E_n^0\atop \|w\|_2=1}
          \int_0^{R_0} (a_0(r,R_0)^2 + a_2(r,R_0)^2) \, r^3\, dr \\
 &\le \hat\mu_n(R_0) - C_4 (R-R_0),
\end{align*}
with constant $C_4>0$ independent of $R$,
using \eqref{N3} and the finite dimensionality of $E_n^0$.  Since $\hat\mu(R)$ is isolated and simple in a neighborhood of $R_0$,  by Kato-Rellich it is differentiable at $R_0$.  By the above estimate, we conclude that $\hat\mu'_n(R_0)\le -C_4 <0$.

\end{proof}

 From Lemma~\ref{step4} we thus have:

\begin{cor}\label{muRneg}
Denote by $\hat\mu_1(R)$ the smallest eigenvalue of $\hat Q_R$.  Then $\hat\mu'_1(R)<0$ for all $R>0$.
\end{cor}

In order to return to the problem on a fixed ball $\mathbb{D}_1$, with parameter $\la=R^2$,  we recall that the eigenvalues of $\mathcal{M}_\la$ and $\widehat{\mathcal{M}}_R$ are related via $\mu_n^{(1)}(\la)=R^2\hat\mu_n(R)$, $\la=R^2$.  So, 
${d\over d\la} \mu_n^{(1)}(\la)= \frac12 R\hat\mu'_n(R) + \hat\mu_n(R)$ is negative at a simple eigenvalue at the point at which $\mu_n(R)=0$ (that is, exactly at a bifurcation point):

\begin{cor}\label{vanish}
Suppose $\mu_n^{(1)}(\la)$ is a simple eigenvalue of $\mathcal{M}_\la$ for $|\la-\la_0|<\delta$, and $\mu_n^{(1)}(\la_0)=0$.  Then, 
${d\over d\la} \mu_n^{(1)}(\la_0)<0$.
\end{cor}

Finally, we show that at least one eigenvalue of $\mathcal{M}_\la$ must cross through zero as $\la$ increases:  the ground state $\mu_1^{(1)}(\la)$.

\begin{lem}\label{step5}  
 There exists a unique $\la_\beta>0$ so that $\mu_1^{(1)}(\la)>0$ for $\la<\la_\beta$ and $\mu_1^{(1)}(\la)<0$ for $\la>\la_\beta$.
\end{lem}

\begin{proof}  By Corollary~\ref{muRneg} above, $\hat\mu_1(R)$ is strictly decreasing.  From Corollary~\ref{vanish}, it suffices to show that $\hat\mu_1(R)<0$ for some sufficiently large $R$. Thus $\hat\mu_1(R_\beta)=0$ at a unique $R_\beta>0$,  whence $\mu_1^{(1)}(\la_\beta)$ crosses through zero at a unique $\la_\beta=R_\beta^2$.  
To do this we argue as in Theorem~2 of \cite{M1}.  As $R\to\infty$, the radial profile
$ F(\cdot,R)\to F_\infty(\cdot)$ in $C^k([0,R])$ for all $R>0$ and $k\in\NN$, with $F_\infty$ the modulus of the unique entire equivariant solution of the form $u_\infty=F_\infty(r) e^{i\theta}$.  We have already shown that for $0<\beta<1$, the entire equivariant solution $U_\infty={1\over\sqrt{2}}(u_\infty,u_\infty)$ is {\em not} a local minimizer (in the sense of de Giorgi) in $\RR^2$, so there exists $R>0$ and $\Phi\in C_0^\infty(\mathbb{D}_R)$ for which 
$E(U_\infty+\Phi; \mathbb{D}_R)< E(U_\infty; \mathbb{D}_R)$.  
By an approximation argument, we could then conclude that $E''_R(U_R)[\Phi]<0$ for some $\Phi\in C_0^\infty(\mathbb{D}_R)$, and hence $\mu_1^{(1)}(R^2)<0$ for that value of $R$.

Instead, we give a more direct proof, using $\hat Q_R$.  Let $\hat a_0(r),\hat a_2(r)$ be the ground-state eigenfunctions.  Define $A=\hat a_0-\hat a_2$ and $B=\hat a_0+ \hat a_2$.   In terms of $A,B$ we have:
\begin{align*}  \hat Q_R(\hat a_0, \hat a_2) &=
4\pi\int_0^R\left\{
       (A')^2 + (B')^2 + {2\over r^2}(B-A)^2
           + (F^2-1)(A^2+B^2) + 2\beta A^2\right\} r\, dr  \\
&=:\breve Q_R(A,B).
\end{align*}
We would like to make the choice $A=F'_\infty$ and $B=F_\infty/r$, but  these functions are not admissible as test functions, since they do not vanish at $r=R$.  Nevertheless, since $F_\infty$ vanishes linearly at $r=0$, $A(r)$ and $B(r)$ are regular near $r=0$, and since $A-B=r(F_\infty/r)'$, the second term in $\breve Q_R$ is well-defined.  Moreover, it is well-known (see \cite{HH}) that the  derivatives of $F_\infty$ decay sufficiently rapidly  as $r\to\infty$ in order to have $A$ and $B$ in the domain of definition of $ Q_\infty$, and 
$$  \breve Q_\infty(A,B)= \lim_{R\to\infty} \breve Q_R(A,B).  $$
This last quantity we can evaluate exactly, using the equations which $A,B$ solve (see (18') of \cite{M1}):
$$  \left\{ \begin{aligned}
&-A'' - {1\over r}A' + {2\over r^2} (A-B) - (1-3F_\infty^2)A =0 \\
&-B'' - {1\over r}B' - {2\over r^2} (A-B) - (1-F_\infty^2)B =0
\end{aligned}
\right. .$$
We multiply the first equation by $A$, the second by $B$, add, and integrate by parts to obtain
$$  \breve Q_\infty(A,B) = 8\pi(\beta-1)\int_0^\infty F_\infty^2A^2\, r\, dr
          < 0,  $$
when $0<\beta<1$.  By approximation, we may find a large $R$ and $\breve A, \breve B\in C^\infty(0,R)$ such that 
$\breve Q_R(\breve A, \breve B) = \breve Q_\infty(\breve A, \breve B)<0$. Thus, for sufficiently large $R$, $\mu_1^{(1)}(R^2)= R^2\hat\mu_1(R)<0$.
\end{proof}

\subsection*{Bifurcation at simple eigenvalues}

We are finally ready to prove bifurcation of symmetric solutions to \eqref{DP} at simple eigenvalues of $L_\la$.
The following lemma summarizes the previous results on the eigenvalues of $L_\la$, obtained in the previous parts:

\begin{prop}\label{summary}
Assume $0<\beta<1$.
\begin{enumerate}
\item[(a)]  Suppose $\mu_n(\la_0)$ is a simple eigenvalue of 
$\mathcal{M}_{\la_0}$.  Then:
\begin{enumerate}
\item[(i)] $\mu_n(\la)$ is an eigenvalue of $L_\la$, with eigenspace
$$   \XL = \left\{ s\, W_\xi: \  \ s\in \RR, \ \xi\in \so\right\},$$
with $W_\xi=(w_\xi, -w_\xi)$, $w_\xi = R_\xi\, w_1 =
 \overline{\xi} \, a_0(r;\la,n) - \xi\, a_2(r;\la,n)\, e^{2i\theta}$.
 \item[(ii)]  Both the eigenvalue $\mu_n(\la)$ and normalized eigenvectors are analytic in a (complex) neighborhood of $\la_0$.
 \item[(iii)] If $\mu_n(\la_0)=0$, then $\mu_n'(\la_0)<0$.
\end{enumerate}
\item[(b)]  There exists a unique $\la_\beta>0$ for which the ground state eigenvalue is given by $\mu_1(\la_\beta)=0=\mu_1^{(1)}(\la_\beta)$, and $0$ is a simple eigenvalue of 
$\mathcal{M}_{\la_\beta}$.
\end{enumerate}
\end{prop}

\begin{proof}
Statement {\it (i)} follows from the reductions in Lemma~\ref{step1} and \eqref{direct}, together with the description of the eigenspaces of $\LL_\la^{(1)}$ in Lemma~\ref{step3}.  The analyticity claimed in {\it (ii)} was proven in Lemma~\ref{step6}, and (iii) follows from Corollary~\ref{vanish}.  Part {\it (b)} puts together the results of Lemmas~\ref{step4} and \ref{step5} with that of Corollary~\ref{cor1}.
\end{proof}

\medskip

We are ready to  analyze the bifurcation of solutions at isolated simple eigenvalues of $\mathcal{M}_\la$ which cross zero at $\la_0>0$.
We write 
$\Psi\in \HH$ as $\Psi=U_\lambda + V$ with
$V=(v_+,v_-)\in H^1_0({\mathbb D_1};\CC^2)$.  Then, $\Psi$ solves \eqref{DP} if and only if
\be\label{FV}
 0= F(V,\lambda):= L_\lambda V + \lambda H(V,\lambda) =
     \left[ \begin{matrix}
L_\lambda^+ V + \lambda H^+(V,\lambda)  
   \\
L_\lambda^- V + \lambda H^-(V,\lambda) 
\end{matrix}\right],
\ee
with
\begin{equation}\label{Hdef}
\left.
\begin{aligned}
H^+(V,\lambda) &= (|v_+|^2+|v_-|^2)(v_+ +{u_\lambda\over\sqrt{2}})
      + 2\llan{u_\lambda\over\sqrt{2}},v_+ + v_-\rran v_+
      \\
      &\qquad
      +\beta\left(|v_+|^2-|v_-|^2\right) 
                (v_+ + {u_\lambda\over\sqrt{2}})
          + 2\beta \llan{u_\lambda\over\sqrt{2}}, v_+ - v_-\rran v_+
          \\
          H^-(V,\lambda)&=
          (|v_+|^2+|v_-|^2)(v_- +{u_\lambda\over\sqrt{2}})
      + 2\llan {u_\lambda\over\sqrt{2}},v_+ + v_-\rran v_-
      \\
      &\qquad
      -\beta\left(|v_+|^2-|v_-|^2\right) 
                (v_- + {u_\lambda\over\sqrt{2}})
          - 2\beta \llan{u_\lambda\over\sqrt{2}}, v_+ - v_-\rran v_-
\end{aligned}
\right\}.
\end{equation}
The above defines a smooth map $F: \ H^1_0({\mathbb D_1}; \CC^2)\times \RR^+\to H^{-1}({\mathbb D_1}; \CC^2)$.  If we identify a complex vector 
$V=(v_+,v_-)=(v^1_+ + i v^2_-, \, v^1_- + i v^2_-)\in \CC^2$ with the real vector $(v^1_+ , v^2_+, \, v^1_- , v^2_-)\in\RR^4$, then by Proposition~\ref{fanalytic} we recognize that the symmetric solution $U_\la$ is real-analytic as an $\RR^4$-valued function of $(x,\la)$, and the map $F$ is likewise real-analytic, viewed as a map of $V\in H^1_0(\ddd;\RR^4)$ and $\la$.
 We also note that
$H_V(0,\lambda) = 0$, and so $F_V(0,\la)=L_\la$, the linearization around the symmetric vortex solution $U_\la$.

We now state our bifurcation result.  Denote by $\XL^\perp$ the orthogonal complement of $\XL$ in $H^1_0(\mathbb{D}_1;\CC^2)$.

\begin{thm}\label{bifthm}  Suppose $\la_0>0$ is such that, for some $n\ge 1$, $\mu_n^{(1)}(\la_0)$ is a simple eigenvalue of $\mathcal{M}_{\la_0}$ with $\mu_n^{(1)}(\la_0)=0$. Then $(U_{\lambda_0},\lambda_0)$ is a point of bifurcation for the equations \eqref{DP}.  In particular:
\begin{enumerate}
\item[(1)]  there exists a neighborhood $\mathcal{N}$ of 
$(U_{\lambda_0},\lambda_0)$ in $\HH\times (0,\infty)$, $\delta>0$, and real analytic maps $\Phi: \ (-\delta,\delta)\to \mathcal{X}_{\la_0}^\perp$ and $\phi: \ (-\delta,\delta)\to (0,\infty)$ with $\Phi(0)=0$, $\phi(0)=\la_0$, such that there exists a non-equivariant solution 
 \eqref{DP} of the form
$$  \Psi (t,\xi) = U_{\phi(t)} + t R_\xi(W_1+ \Phi(t)), \quad \la=\phi(t) $$
for all $|t|<\delta$ and $\xi\in \mathbb{S}^1$.  Moreover, $\Psi (-t,\xi)=R_{-1}\Psi (t,\xi)$ and $\phi(-t)=\phi(t)$.
\item[(2)]  Any solution $(\Psi,\la)$ of \eqref{DP} in the neighborhood $\mathcal{N}$ is either an equivariant solution $(U_\la,\la)$ or of the form $(\Psi (t,\xi),\phi(t))$ above.
\item[(3)]  All solutions of \eqref{DP} in $\mathcal{N}$ satisfy $T\Psi=\Psi$ (where the involution $T$ is defined in \eqref{T}.)  Moreover, each component $\psi_\pm(t,\xi)$ of  $\Psi (t,\xi)=(\psi_+(t,\xi),\psi_-(t,\xi))$ has exactly one zero, and their zeros are antipodal and distinct from the origin.
\end{enumerate}
\end{thm}

By {\it (b)} of Proposition~\ref{summary}, we may apply Theorem~\ref{bifthm} at $\la=\la_\beta$, and obtain bifurcation at the ground state eigenvalue $\mu_1(\la_\beta)=0$, which implies Theorem~\ref{thm3} stated in the Introduction.

\begin{proof}
Given any $V\in H^1_0(\ddd;\CC^2)$, there exists a unique $\alpha>0$,  $\xi\in\so$, and $\tilde Z\in \XXo^\perp$ with 
$V=\alpha W_\xi + \tilde Z= R_\xi\, (\alpha W_1 + Z)$, with $Z=R_{\overline\xi}\tilde Z\in\XXo^\perp$ and $W_1=W_1(\la)$ as in {\it (i)}  of Proposition~\ref{summary}, normalized with $\|W_1\|_{L^2}=1$.  Define the Hilbert space
$\mathcal{Y}= \{\alpha W_1\}\otimes \XXo^\perp$,
as a subspace of $H^1_0(\ddd;\CC^2)$.  We then consider the equation $F(V,\la)=0$ as in \eqref{FV}, restricted to $V\in\mathcal{Y}$.  Since $F$ is equivariant under the $R_\xi$ action, every solution $V\in \mathcal{Y}$ gives rise to an orbit of solutions in $H^1_0(\ddd;\CC^2)$, while any solution $V\in H^1_0(\ddd;\CC^2)$ corresponds to a solution $R_\xi V\in\mathcal{Y}$ by an appropriate choice of $\xi\in\so$.  Thus, it suffices to consider the equation \eqref{FV} in the smaller space $\mathcal{Y}$ to determine the solution space in $H^1_0(\ddd;\CC^2)$.

We next list some properties of $F$ restricted to $\mathcal{Y}$.  First, it remains true that $F(0,\la)=0$ for all $\la>0$, and $F$ is a real analytic function of $V\in\mathcal{Y}$, $\la>0$ (by Proposition~\ref{summary} {\it (ii)}  and Proposition~\ref{fanalytic}, thinking of $U_\la,V\in \CC^2$ as real vectors in $\RR^4$.)  By the restriction to $\mathcal{Y}$, we have
$F_V (0,\la_0) = L_{\la_0}$ with $\text{ker}_{\mathcal{Y}}(L_{\la_0}) = \{\alpha W_1\}$, and thus $\dim\text{ker}_{\mathcal{Y}}(L_{\la_0})=1=\text{codim}\, \text{Ran}(L_{\la_0})$.  Lastly, we calculate the derivative $\mu'_n(\la_0)$ in terms of the function $F$:  as $\mu_n(\la)=\langle W_1(\la), F_V(0,\la)\, W_1(\la)\rangle,$ by {\it (iii)}  of Proposition~\ref{summary} we have
\begin{align} \nnn
0>\mu'_n(\la_0) &= 2\langle W'_1(\la_0), \, F_V(0,\la_0) W_1(\la_0)\rangle
    + \langle W_1, F_{V,\la}(0,\la_0) W_1\rangle  \\
    & = \langle W_1, F_{V,\la}(0,\la_0) W_1\rangle, \label{CRcond}
\end{align}
as $F_V(0,\la_0)W_1=L_{\la_0}W_1=0$ (since $\mu_n(\la_0)=0$).  
We now claim that $F_{V,\la}(0,\la_0)W_1(\la_0)\not\in\text{Ran}_{\mathcal{Y}}\,(L_{\la_0})$.  Indeed, assume the contrary, so there exists $X\in\mathcal{Y}$ with $F_V(0,\la_0)X= F_{V,\la}(0,\la_0)W_1(\la_0)$, and take the scalar product with $W_1(\la_0)$.  We have
$$  0>\mu'_n(\la_0)= \langle W_1(\la_0), F_V(0,\la_0) X\rangle = 0, $$
a contradiction.

The celebrated Crandall--Rabinowitz bifurcation theorem (Theorem~1.7 of \cite{CR}) may then be applied to $F$ in the space $\mathcal{Y}$.  We note that since $F$ is an analytic map, by invoking the analytic version of the Implicit Function theorem in the proof of \cite{CR} the maps obtained will be real analytic.  We conclude that there exists a neighborhood $\tilde{\mathcal{N}}$ of $(0,\la_0)$ in $\mathcal{Y}\times(0,\infty)$, $\delta>0$, and real analytic maps $\Phi$, $\phi$ as in the statement of the theorem, such that 
$$   F^{-1}\{0\} \cap \mathcal{\tilde N} =
  \biggl\{ (V^t,\la)=\left( t[W_1(\la_0)+\Phi(t)], \phi(t)\right): \ |t|<\delta\biggr\} \cup \left\{ (0,\la)\in \tilde{\mathcal{N}}\right\}.  $$
Define the neighborhood $\mathcal{N}\in H^1_0(\ddd;\CC^2)$ as the union of the images of $\mathcal{\tilde N}$ under the action of $R_\xi$, $\xi\in \so$. 
The characterization of the solution set in $\mathcal{N}$ in statements {\it (1)}  and {\it (2)}  then follows.

It remains to verify the symmetry results in {\it (1)}  and {\it (3)}.  First, we note that if $(V^t,\phi(t))$ is a solution in $\mathcal{\tilde N}$ of the above form, then
$$  \tilde V:= R_{-1}V^t = (-t)(W_1 -\Phi(t))\in \mathcal{\tilde N}\subset\mathcal{Y} $$
is also a non-equivariant solution with $\la=\phi(t)$.  Therefore, there exists $s$, $|s|<\delta$ for which $(\tilde V,\phi(t))=(V^s,\phi(s))$.
Since $W_1\perp \XXo^\perp$, we have
$$   (s+t) W_1 = s\Phi(s) - t\Phi(t) = 0, $$
and hence $s=-t$ and $\Phi(s)=\Phi(-t)=-\Phi(t)$.  We conclude that $\phi(-t)=\phi(t)$ and $V^{-t}=R_{-1}V^t$, which finishes the proof of {\it (1)}.

Next, we show $TV^t=V^t$ for the involution $T$.  As above, $TV^t$ is also a solution with the same $\la=\phi(t)$, in the neighborhood $\mathcal{\tilde N}$.  We note that $TW_1=W_1$, and by following the same arguments as above, $T\Phi(t)=\Phi(t)$.  Since $TU_\la=U_\la$, we conclude $T\Psi (t,\xi)=\Psi (t,\xi)$ for the whole family of solutions.

Finally, we consider the zero set of each component of $\Psi (t,\xi)$.  
Since $\Psi (t,\xi)\to U_{\la_0}$ in $C^2$ as $t\to 0$, and the equivariant solutions $U_{\la_0}$ have exactly one non-degenerate zero (the origin) in each component, the same must be true for $\Psi (t,\xi)$ for $|t|$ sufficiently small.  Since $\Psi (t,\xi)$ is fixed by the involution $T$, the zeros of the components $\psi_+(t,\xi)$, $\psi_-(t,\xi)$ must be antipodal.  By 
Proposition~\ref{unique2}, only the equivariant solution $U_\la$ vanishes in both components at the origin, so the zeros of $\psi^t_\pm$ must be antipodal and distinct.
This concludes the proof of Theorem~\ref{bifthm}.
\end{proof}

\begin{rem}\label{CR}\rm
(a) \ Given that the solution curves are analytic, we may expand them around the bifurcation point $\la_0$ and (in principle) obtain further information about the direction and stability of the bifurcating solutions.  For instance, we may calculate higher derivatives of $\la=\phi(t)$, and obtain $\la'=\phi'(0)=0$, and 
$$  \lambda''=\phi''(0) = {\displaystyle\int_{\ddd} \langle L_{\lambda_\beta} V''(0),  V''(0)\rangle - 2\lambda_\beta\int_{\ddd} |W_1|^4 \over \mu'(\lambda_\beta)}. $$
Since $L_{\lambda_\beta}$ is positive definite on $\XXo^\perp$, the  sign of the numerator is not clear {\it a priori}, so numerical approximation may be necessary to determine the details of the bifurcation at $\lambda_\beta$.

\medskip
\noindent
(b) \ In a similar vein, if we compute the quantity in \eqref{CRcond} directly, we obtain
$$  \int_{\ddd} \llan F_{V,\lambda}(0,\lambda_0)W , W \rran  
  = 4\pi \int_0^1 \left[
  \partial_\lambda\left( \lambda (f_\la^2-1)\right) [a_0^2 + a_2^2]
    + \partial_\lambda\left(\lambda f_\la^2\right) [a_0-a_2]^2
  \right] r\, dr \biggr|_{\lambda=\lambda_0}.
$$
Expressed in this form, it is not apparent whether this quantity is non-zero.  Only by recognizing the connection to the derivative $\mu'_n(\la)$ are we able to apply the Crandall-Rabinowitz theorem.

\noindent
(c) \  While we know (from Lemma~\ref{step5}) that the ground state eigenvalue $\mu_1(\la)$ must cross zero for $\beta\in (0,1)$, it is unclear whether any of the higher eigenvalues can lead to other bifurcations of the symmetric solutions.

\medskip

\noindent
(d) \ A very general result by Rabinowitz \cite{R} shows that bifurcation always occurs at eigenvalues of any finite multiplicity in a variational problem.  Although the form of equation assumed in \cite{R} is somewhat different than our $F(V,\la)$, the result nevertheless holds true in our setting, although the conclusions of the bifurcation theorem are weaker than the statement obtained by using the simplicity of the eigenspace as in \cite{CR}.  In particular, one may conclude that non-equivariant solutions exist in any neighborhood of $U_{\la_0}$ when $0=\mu_n(\la_0)$ is a degenerate eigenvalue, and by analyticity (see \cite{Loj}) the continua of solutions form finitely many analytic curves, but there is no complete characterization of the solution set as in {\it (1)}, {\it (2)} of Theorem~\ref{bifthm}. 
\end{rem}

\end{document}